\documentclass[11pt, a4paper,reqno]{amsart}
\usepackage{amsmath,amsthm,amscd,amssymb,amsfonts, amsbsy}
\usepackage{epigraph}
\usepackage[english]{babel}
\usepackage{latexsym} 
\usepackage{txfonts}
\usepackage{exscale}
\usepackage{enumitem}
\usepackage{bbm}
\usepackage{enumitem}
\usepackage{graphicx}
\usepackage{cite}
\usepackage[colorlinks,citecolor=red,hypertexnames=false]{hyperref}
\usepackage{color}
\newcommand{\noop}[1]{}
\usepackage[textsize=small]{todonotes}
\calclayout
\allowdisplaybreaks

\theoremstyle{plain}
\newtheorem{theorem}[equation]{Theorem}
\newtheorem{lemma}[equation]{Lemma}
\newtheorem{corollary}[equation]{Corollary}
\newtheorem{proposition}[equation]{Proposition}

\theoremstyle{definition}
\newtheorem{definition}[equation]{Definition}

\theoremstyle{remark}
\newtheorem{remark}[equation]{Remark}

\numberwithin{equation}{section}

\newcommand{\N}{\mathbb{N}}

\renewcommand{\Re}{\operatorname{Re}}

\newcommand{\eps}{\varepsilon}

\newcommand{\re}{\mathbb{R}}
\newcommand{\R}{\mathbb{R}}

\newcommand{\rn}{{\mathbb{R}^n}}

\newcommand{\ree}{\mathbb{R}^{1+n}}

\newcommand{\cL}{\mathcal{L}}

\newcommand{\lamprime}{\varrho} %A lambda that doesnt look like one
\newcommand{\bigdot}{\boldsymbol{\cdot}}
\DeclareMathOperator{\I}{I}
\DeclareMathOperator{\II}{II}
\DeclareMathOperator{\III}{III}

\def\Xint#1{\mathchoice
{\XXint\displaystyle\textstyle{#1}}%
{\XXint\textstyle\scriptstyle{#1}}%
{\XXint\scriptstyle\scriptscriptstyle{#1}}%
{\XXint\scriptscriptstyle%
\scriptscriptstyle{#1}}%
\!\int}
\def\XXint#1#2#3{{\setbox0=\hbox{$#1{#2#3}{%
\int}$ }
\vcenter{\hbox{$#2#3$ }}\kern-.6\wd0}}
\def\fint{\,\Xint -} % \, corrects the \! used in the definition
\def\fiint{\fint_{} \kern-.4em \fint}
\def\fiiint{\fiint_{} \kern-.4em \fint}
\renewcommand{\iint}{\int_{}\kern-.34em \int} %\ minor space between the integrals
\renewcommand{\iiint}{\iint_{}\kern-.34em \int} %\ minor space between the integrals

\def\div{\mathop{\operatorname{div}}\nolimits}

\providecommand{\ab}[1]{  \lvert  #1  \rvert }

\providecommand{\no}[1]{  \lVert  #1  \rVert }

\providecommand{\fint}{\strokedint}
\providecommand{\fiint}{\strokediint}
\providecommand{\fiiint}{\strokediiint}

\providecommand{\la}[1]{\langle#1\rangle}

\author{Simon Bortz}
\author{Moritz Egert}
\author{Olli Saari}

\address{Simon Bortz, Department of Mathematics, University of Alabama, Tuscaloosa, AL, 35487, USA}
	\email{sbortz@ua.edu} 

\address{Moritz Egert, Department of Mathematics, TU Darmstadt, Schlossgartenstra\ss e 7, 64289 Darmstadt, Germany}
\email{egert@mathematik.tu-darmstadt.de}

\address{Olli Saari, Departament de Matem\`atiques,
	Universitat Polit\`ecnica de Catalunya,
	Avinguda Diagonal 647, 08028 Barcelona,
	Catalunya, Spain}

\address{Institute of Mathematics of UPC-BarcelonaTech, Pau Gargallo 14, 08028 Barcelona, Catalunya, Spain}

\address{Centre de Recerca Matem\`atica, Edifici C, Campus Bellaterra, 08193 Bellaterra, Catalunya, Spain}

\email{olli.saari@upc.edu}

\begin{document}
\allowdisplaybreaks
%%%%%%%%%%%%%%%%%%%%%%%%%%%%%%%%%%%%%%%%%%%%%%%%%%%%%%%%%%%%%%%%
\title[A stability result for parabolic measures]{A stability result for parabolic measures of operators with singular drifts}
%%%%%%%%%%%%%%%%%%%%%%%%%%%%%%%%%%%%%%%%%%%%%%%%%%%%%%%%%%%%%%%%
\begin{abstract}
We study the operator 
\[ \partial_t -  \div A \nabla + B \cdot \nabla \]
in parabolic upper-half-space, where $A$ is an elliptic matrix satisfying an oscillation condition 
and $B$ is a singular drift with a Carleson control. Our main result establishes quantitative $A_{\infty}$-estimates for the parabolic measure 
in terms of oscillation of $A$ and smallness of $B$.
The proof relies on new estimates for parabolic Green functions that quantify their deviations from linear functions of the normal variable and on a novel, quantitative Carleson measure criterion for anisotropic $A_{\infty}$-weights.%, which is of independent interest.
%As intermediate results, we provide parabolic versions of Green function
%estimates due to David, Li and Mayboroda and a new, quantitative Carleson measure criterion for anisotropic $A_{\infty}$-weights.
\end{abstract}
%%%%%%%%%%%%%%%%%%%%%%%%%%%%%%%%%%%%%%%%%%%%%%%%%%%%%%%%%%%%%%%%
%\thanks{}
\maketitle
\tableofcontents
%%%%%%%%%%%%%%%%%%%%%%%%%%%%%%%%%%%%%%%%%%%%%%%%%%%%%%%%%%%%%%%%
\section{Introduction}

\noindent We work in the parabolic upper-half-space 
\[
\mathbb{R}^{n+1}_+ \coloneqq \{(t,x,\lambda): t \in \mathbb{R}, x \in \mathbb{R}^{n-1}, \lambda > 0\}
\]
and study differential operators of the form 
\begin{equation}\label{Ldefine.eq}
\cL \coloneqq \partial_t -  \div A \nabla + B \cdot \nabla,
\end{equation}
where $A : \mathbb{R}^{n+1}_+ \to \mathbb{R}^{n \times n}$ is an elliptic matrix-valued function (with ellipticity constant $M_0$) and $B: \mathbb{R}^{n+1}_+ \to \mathbb{R}^{n}$ is a singular drift term with a uniform bound $|\lambda B(t,x,\lambda)| \leq \varepsilon_0$. The matrices $A$ satisfy a (weak) Dahlberg--Kenig--Pipher oscillation condition in all variables and $\lambda B$ satisfies a smallness condition, both being measured by certain coefficients $\alpha_{A,B}(t,x,\lambda)$ and quantified through the requirement that
\[
d\nu_{A,B}(t,x,\lambda) \coloneqq  \alpha_{A,B}(t,x,\lambda)^{2}  \frac{dt\, dx\, d\lambda}{\lambda}
\]
should be a (parabolic) Carleson measure on $\ree_{+}$, see Section~\ref{Preliminaries} for precise definitions and further notation. This class of operators arises as pull-backs when studying parabolic equations in parabolic Lipschitz domains~\cite{HL-Mem}. 

Our main theorem is a stability result for the parabolic measure, analogous to 
the culmination of the sequence of the three papers \cite{DLM}, \cite{BTZ2} and \cite{BES-DKP} from the elliptic setting.
Moreover, it can be viewed as an analytic version of a small constant, variable coefficients variant of \cite{MR2053754}. It quantitatively demonstrates that small oscillations of $A$ and small drift $B$ lead to small oscillations of the parabolic measure when the pole is sufficiently far from the boundary.
%%%%%%%%%%%%%%%%%%%%%%%%%%%%%%%%%%%%%%%%%%%%%%%%%%%%%%%%%%%%%%%%
\begin{theorem}
\label{main-theorem}
Let $M_0$ be given.
There are $\varepsilon_0 > 0$, $\delta_0 > 0$ and $C \ge 1$ such that the following hold whenever $\cL$ is $(M_0,\varepsilon_0)$-parabolic and $(t_0,x_0,\lambda_0) \in \ree_{+}$.
\begin{enumerate}
	\item For every $\delta \in (0, \delta_0)$ there is $\kappa \geq 40$ such that if $\no{\nu_{A,B}}_{\mathcal{C}} \leq \delta$, then the $\cL$-parabolic measure $\omega$ with pole at the forward corkscrew point $a^{+}(t_0,x_0,\kappa\lambda_0)$ is absolutely continuous with respect to Lebesgue measure on $Q^n(t_0,x_0,2\lambda_0)$.
	\item Denoting $k \coloneqq \frac{d\omega}{dt\,dx}$ in the setting of (1), the quantitative $A_\infty$-estimate
	\[
	\log \biggl(\fiint_{Q} k(t,x) \, dt \, dx \biggr)  - \fiint_{Q} \log k(t,x) \, dt \, dx 
	\le C \sqrt{\delta} 
	\]
	holds for all parabolic cubes $Q \subset Q^n(t_0,x_0,\lambda_0)$.
\end{enumerate}
\end{theorem}
%%%%%%%%%%%%%%%%%%%%%%%%%%%%%%%%%%%%%%%%%%%%%%%%%%%%%%%%%%%%%%%%
The supremum of the left-hand side in part (2) of the theorem can be taken as a definition of the (local) $A_\infty$-constant of $\omega$. It implies the following quantitative statement for $L^p$-solvability of the Dirichlet problem for $\cL$, see Section~\ref{main results.sec}.

\begin{corollary}
\label{main-cor}
Let $M_0$ be given.
There is $\varepsilon_0 > 0$ such that if $\cL$ is $(M_0,\varepsilon_0)$-parabolic, then for every $q \in (1,\infty)$ there exists $\delta > 0$ such that
$\no{\nu}_{\mathcal{C}} \le \delta$ implies that all parabolic measures in Theorem~\ref{main-theorem} are locally reverse  H\"older weights with exponent $q$ and uniform constants. In particular, the $L^p$-Dirichlet problem for $\cL$ in the upper-half-space is solvable, where $1/p + 1/q = 1$.
\end{corollary}

\begin{remark}
The methods here also seem to apply for $\mathcal{L}$ with no drift but without the smallness assumption on the measure $\nu$. In that case, one expects to obtain that the associated parabolic measure belongs to $A_\infty$ without quantitative control on the constant. The only modification required is to replace the proof of Theorem \ref{green-main-estimate.thm} with its ``large constant” analogue closely following \cite[Section 4]{DLM}. This result, in itself, would be new because of the {\bf weak}-DKP condition, but is not the focus of the current work (small constant results).
\end{remark}

\subsection{Strategy}The three main steps of the proof are 
\begin{enumerate}
	\item[(a)] an estimate on deviation of the Green function from a linear function,
	\item[(b)] an estimate on approximations to distributional derivatives of the parabolic measures,
	\item [(c)]and a local and quantitative criterion for anisotropic $A_{\infty}$-weights in terms of Carleson conditions. 
\end{enumerate}
This general idea corresponds to the sequence of papers \cite{DLM}, \cite{BTZ2} and \cite{BES-DKP} from the elliptic setting;
but the first and the third part of the proof exhibit significant differences 
due to the evolutionary nature of the problem.

Part (a) is carried out in Sections \ref{Estimates for constant coefficient operators} -
\ref{Estimates for Green function}.
The main difficulty is the time-directedness of the Harnack inequality in the context of general solutions. Already the proof for non-degeneracy of the Dirichlet energy 
for constant coefficient operators in Lemma~\ref{dec-lb-jande.lem} requires a completely novel argument and the perturbation argument for variable coefficients in Lemma~\ref{diffOK.lem} would have not been possible, had we been not able to sweep away one of the essential technical difficulties pointed out in \cite{DLM}.
Our proof could hence be used to improve even the elliptic estimate (compare to Lemma~3.19 in \cite{DLM}).

Part (b) consists in using the Riesz formula for solutions in order to bound approximants to distributional derivatives of the parabolic measure 
by quantities related to the Green function. This is carried out in Section \ref{Estimates for the parabolic measure}
and follows the lines of thought of \cite{BTZ2} with few modifications.

Part (c) constitutes our second main result. It is presented in Section \ref{anisotropic weights} and should be useful in different contexts. Our Carleson condition is quantitative in the spirit of \cite{BES-DKP} but without global doubling as a background hypothesis on the weight. 
Providing a `Carleson test' that is amenable for weights that are doubling only locally is important, as parabolic measures fall precisely into this class. To this end, we renounce the algebraic properties of the heat kernel in  \cite{FKP, BES-DKP}
to be able to directly work with approximations of identity based on compactly supported test functions.
Even though the results in \cite{FKP} also include a characterization based on compactly supported test functions,
their argument still uses representation through heat kernels and thus global doubling seems to be necessary.

\subsection{Background on parabolic measures}

Let us provide some additional context from the extensive literature related to parabolic measure. 

Based on Dahlberg's result~\cite{DahlbergRH2} that harmonic measure above a Lipschitz graph satisfies a reverse H\"older inequality with exponent $p=2$, R. Hunt conjectured (see \cite[p.~2]{MR563542}) that caloric measure above a parabolic Lipschitz graph is locally an $A_\infty$-weight and therefore absolutely continuous with respect to surface measure. This turned out to be untrue \cite{KW}. Only very recently, it was shown that caloric measure above a parabolic Lipschitz graph is in $A_\infty$ if and only if the function defining the graph has a half-order time derivative in the parabolic $BMO$-space \cite{BHMN}. One direction was known for some time and is due to Lewis and Murray \cite{LewMur} and the same condition on the half-order time derivative had already been identified as a natural one from the context of parabolic singular integrals, boundary value problems and Layer potentials in \cite{Hof-Commchar, H-SIO-Duke, HL-Ann, HL-JFA, Mur-comm}. Moreover, the true analog of Dahlberg's result (density with reverse H\"older exponent $p=2$) requires smallness of the half-order time derivative \cite{HL-Ann}. 

Contrarily to the elliptic setting~\cite{JK-sym}, the most obvious flattening pull-back from a parabolic graph domain to the half-space cannot be used because it fails to be Lipschitz in the $t$-variable, and the more elaborate smoothed pullback of Dahlberg--Kenig--Stein~\cite{Dahlpullback} introduces a singular drift term, even when starting from the heat equation~ \cite{HL-Ann,HL-Mem,HL-JFA}. Moreover, if in the situation of the Dahlberg--Kenig--Stein pullback the half-order time derivative for the graph function is in parabolic $BMO$ and the original coefficients $A$ and $B$ in the graph domain satisfy the (strong) Dahlberg--Kenig--Pipher condition (supremum of $|\partial_t A|^2 \delta^3$, $|\nabla A|^2 \delta$ and $|B|^2 \delta$ on Whitney regions are all Carleson measures~\cite{HL-JFA}), then so do the pulled-back coefficients. In particular, the pull-backs are operators of type \eqref{Ldefine.eq} with a stronger oscillation condition than in the present paper.

For operators of type $\cL$ in the upper-half-space, Hofmann and Lewis \cite{HL-Mem} proved that the (strong) Dahlberg--Kenig--Pipher condition implies that the associated parabolic measure is an $A_\infty$-weight. Subsequently, $L^p$-solvability of the Dirichlet problem for $p$ close or equal to $1$ in parabolic graph domains (with small half order time derivative measured in $BMO$) were shown by Dindo\u{s}--Petermichl--Pipher \cite{DPP} and for $p$ close to $1$ in rougher domains by Dindo\u{s}--Dyer--Hwang \cite{DDH}. 

In this regard, the main contribution of the present work is the fist qualitative small constant $A_\infty$-results and the use of the (weak  a.k.a $L^2$) Dahlberg--Kenig--Pipher condition on the coefficients.
We conjecture that a direct application of our results here (along with tracking constants) will give small constant $A_\infty$-results in the graph case and that --- less directly and following the work of \cite{DLM2} ---  it may be possible to obtain small constant $A_\infty$-results in rougher settings for operators satisfying a weak Dahlberg-Kenig-Pipher condition.

\noindent \textbf{Acknowledgment.}
The first author is supported by the Simons Foundation's Travel Support for Mathematicians MPS-TSM-00959861.
The second author was supported by Generalitat de Catalunya through the grant 2021-SGR-00087 and by the Spanish State Research Agency MCIN/AEI/10.13039/501100011033, Next Generation EU and by ERDF “A way of making Europe” through the grants RYC2021-032950-I, RED2022-134784-T, PID2021-123903NB-I00 and the Severo Ochoa and Maria de Maeztu Program for Centers and Units of Excellence in R\&D, grant number CEX2020-001084-M.
The authors would like to thank Steve Hofmann for some helpful conversations concerning \cite{HL-Mem}.
%%%%%%%%%%%%%%%%%%%%%%%%%%%%%%%%%%%%%%%%%%%%%%%%%%%%%%%%%%%%%%%%
\section{Preliminaries}
\label{Preliminaries}
%%%%%%%%%%%%%%%%%%%%%%%%%%%%%%%%%%%%%%%%%%%%%%%%%%%%%%%%%%%%%%%%
\subsection{General notation}\label{General notation.sec}

Throughout, we work in space time $\ree$ with $n \in \mathbb{N}$ satisfying $n \ge 2$.  For notational convenience, given any two integers $n_1,n_2 \ge 1$, we identify $\re^{n_1} \times \re^{n_2}$ with $\re^{n_1+n_2}$ so that 
\[
\ree \coloneqq \bigl\{ (t,x,\lambda): t \in \R, \ x \in \R^{n-1}, \ \lambda \in \R \bigr\} .
\] 
We define the parabolic distances by setting for $z,w \in \R^{m}$ with $m \ge 2$,
\[
d_m(z,w) \coloneqq \max \bigl( |z_1-w_1|^{1/2} , \max_{1 < i \le m} |z_i-w_i| \bigr).
\]
This definition will be used with $m \in \{n,n+1\}$. Given $z \in \ree$ and $r > 0$, we define the parabolic cylinder with radius $r$ and center $x$ as 
\[
Q(z,r) \coloneqq \bigl\{ w \in \ree:   d_{n+1}(z,w) < r \bigr\} .
\]
The parabolic boundary of $Q(z,r)$ is 
\[
\partial_{par} Q(z,r) \coloneqq \big \{w \in Q(z,r) : d_{1+n}(z,w)= r, w_1 \neq z_1 + r^2 \big \}
\]
and likewise the adjoint parabolic boundary $\partial_{par^*} Q(z,r)$ is obtained from the topological boundary by excluding the initial boundary $w_1 = z_1 - r^2$.
If $E$ is a measurable set with Lebesgue measure $|E| \in (0,\infty)$ and $f \in L^{1}_{loc}(\ree_{+})$,
we write
\[
\la{f}_{E} \coloneqq \fint_{E} f(z) \, dz \coloneqq \frac{1}{|E|} \int_E f(z)  \, dz.
\]
Further, given a time parameter $\tau$ and a set $E \subset \ree$, we denote the time slice by 
\[
E^{\tau} \coloneqq \bigl \{ (t,x,\lambda) \in E :  t = \tau \bigr \} .
\]
%%%%%%%%%%%%%%%%%%%%%%%%%%%%%%%%%%%%%%%%%%%%%%%%%%%%%%%%%%%%%%%%
\subsection{Point-set correspondence on the upper-half-space}

Our focus will be on solutions to parabolic equations in subsets of the upper-half-space $\ree_+$.
The points $(t,x,\lambda) \in \R\times \R^{n-1} \times (0,\infty) = \ree_+$ 
are set in a one-to-one correspondence with the following geometric objects.
\begin{enumerate}
\item The $n$-dimensional parabolic cylinder or boundary cylinder is defined as 
\[
Q^{n}(t,x,\lambda) \coloneqq \bigl \{(s,y,0) \in \re^{1+n}: d_n((t,x),(s,y)) < \lambda \bigr \}.
\]
\item The parabolic Carleson region is defined as
\[
R(t,x,\lambda) \coloneqq Q^{n}(t,x,\lambda) \times (0,2\lambda).
\]
\item The parabolic Whitney region is defined as
\[
W(t,x,\lambda) \coloneqq Q^{n}(t,x,\lambda) \times (\lambda,2\lambda).
\] 
\item The forward and backward (in time) corkscrew points are defined as
\[
a^{\pm}(t,x,\lambda) \coloneqq (t\pm 2\lambda^{2},x,2\lambda). 
\]
\end{enumerate}

Given $(t,x,\lambda) \in \ree_{+}$ and symbols $Y,Z$,
we sometimes write 
\[
Y(Z(t,x,\lambda)) = Y(t,x,\lambda) .
\] 
to ease notation. Useful examples for this notation are the parabolic Carleson box $R(Q^n(t,x,\lambda)) = R(t,x,\lambda)$ above the boundary cylinder $Q^n(t,x,\lambda)$ and the corkscrew points $a^\pm(R) = a^{\pm}(R(t,x,\lambda)) = a(t,x,\lambda)$. For regions $Z=Z(t,x,\lambda)$ we also write $\theta B \coloneqq Z(t,x,\theta \lambda)$, $\theta>0$, for re-scaled regions of the same type.
%%%%%%%%%%%%%%%%%%%%%%%%%%%%%%%%%%%%%%%%%%%%%%%%%%%%%%%%%%%%%%%%
\subsection{Structure of the equation}

We denote by $\nabla$ and $\div$ the gradient and divergence with respect to all variables except for the time variable.We denote the derivative with respect to the time variable by $\partial_t$.

Let $M_0 \ge 1$. A measurable function $A : \ree_{+} \to \re^{n \times n}$ is said to be $M_0$-elliptic if 
for all $z \in \ree_{+}$ and $\xi \in \rn$, 
\[ A(z) \xi \cdot \xi  \ge M_0^{-1} |\xi|^2 , \quad |A(z) \xi| \le M_0 |\xi| . \]
Let $\varepsilon_0 > 0$. A function $B : \ree_{+} \to \rn$ is said to be an $\varepsilon_0$-small drift if 
for all $z \in \ree_{+}$,
\[ z_{n+1} |B(z)| \le \varepsilon_0,\]
We call 
\[\cL = \partial_t - \div(A \nabla \, \bigdot) + B \cdot \nabla\]
an $(M_0,\varepsilon_0)$-parabolic operator if $A$ is $M_0$-elliptic and $B$ is $\varepsilon_0$-small. The adjoint of an $(M_0,\varepsilon_0)$-parabolic operator is formally written as 
\[
\cL^{*} = -\partial_t - \div (A^* \nabla \, \bigdot) - \div (B \, \bigdot).
\]
Once the parameters $M_0$ and $\varepsilon_0$ are specified,
we call $\cL$ and $\cL^{*}$ briefly parabolic and adjoint parabolic operators.
If $\varepsilon_0 = 0$,  we call the $(M_0,\varepsilon_0)$-parabolic operator drift free.
%%%%%%%%%%%%%%%%%%%%%%%%%%%%%%%%%%%%%%%%%%%%%%%%%%%%%%%%%%%%%%%%
\subsection{Weak solutions}\label{subsec.weak.sol}

Let $\cL$ be an $(M_0,\varepsilon_0)$-parabolic operator. Given an open cube $Q  \subset \ree_+$ and source terms $f \in L_{loc}^2(Q)$ and $F \in L_{loc}^2(Q; \R^n)$, we say that $u$ is a weak solution to $\cL u = f + \div F$ in $Q$ if $u  \in L_{loc}^2(I ; W_{loc}^{1,2}( \Sigma)) \cap C_{loc}(I ; L_{loc}^2( \Sigma))$, where $Q = I \times \Sigma$ with $I$ corresponding to the time variable, and 
\begin{align*}
	\int_Q - u \partial_t \varphi + A \nabla u \cdot \nabla \varphi + (B \cdot \nabla u) \varphi \, dz = 0 \qquad (\varphi \in C_c^\infty(Q)).
\end{align*}
The adjoint equation $\cL^* u = f + \div F$ has an analogous interpretation.  We say that $u$ is a global weak solution in $Q$ if the `loc-subscripts' above can be dropped.  

The definitions of weak solutions follow the classical references \cite{MR0435594, HL-Mem}.  Since $B$ is $\varepsilon_0$-small, we can also write $Bu, B \cdot \nabla u  \in L_{loc}^2(Q)$ as source terms on the right-hand side and conclude from \cite[Theorem~4.2]{ABES_JMPA} that for local solutions the a priori requirement $C_{loc}(I ; L_{loc}^2( \Sigma))$ is not needed but follows from the equation. 

If $Q = R(t,x,\lambda)$ is a Carleson region, we usually work with global weak solutions to $\cL u = 0$ or $\cL^* u = 0$ that satisfy `$u=0$ on $Q^n(t,x, \lambda)$'. This should be understood as $u(s,\bigdot) = 0$ in the Sobolev sense for a.e.\ $s \in I$. In this scenario, for adjoint solutions, we have $Bu \in L^2(Q)$, since the one-dimensional Hardy inequality yields
\begin{align*}
	\int_Q |B u|^2 \, d z 
	\leq \eps_0 \iint_{Q^n(t,x,\lambda)}  \int_0^{\lambda^2} \frac{|u(s,y,\mu)|^2}{\mu^2} \, d \mu  \, d s \, d y
	\leq 4 \varepsilon_0	\int_Q |\partial_\lambda u|^2 \, d z .
\end{align*}
This observation will be used frequently in the following. Sometimes we impose the stronger condition that $u$ is continuous up to $Q^n(t,x, \lambda)$ and vanishes thereon; `$u$ vanishes continuously on $Q^n(t,x,\lambda)$' for short. 
%%%%%%%%%%%%%%%%%%%%%%%%%%%%%%%%%%%%%%%%%%%%%%%%%%%%%%%%%%%%%%%%
\subsection{Carleson measures and energies}
We impose additional smoothness on the coefficients of the differential operators in terms of a Carleson measure condition on the oscillation of $A$ and the size of $B$.
%%%%%%%%%%%%%%%%%%%%%%%%%%%%%%%%%%%%%%%%%%%%%%%%%%%%%%%%%%%%%%%%
\begin{definition}[Carleson Measure]
Given a measure $\mu$ on $\ree_{+}$ and a set $E \subset \ree_{+}$, we write 
\[
\no{\mu}_{\mathcal{C}(E)} \coloneqq \sup_{R' \subset E} \frac{\mu(R')}{|Q^{n}(R')|},
\]
where the supremum is taken over all parabolic Carleson regions $R' \subset E$.
If $\no{\mu}_{\mathcal{C}(E)}< \infty$, we call $\mu$ a parabolic Carleson measure in $E$ and $\no{\mu}_{\mathcal{C}(E)}$ its parabolic Carleson constant or norm.
If $E = \ree_{+}$, then we plainly call $\mu$ a parabolic Carleson measure. 
\end{definition}
%%%%%%%%%%%%%%%%%%%%%%%%%%%%%%%%%%%%%%%%%%%%%%%%%%%%%%%%%%%%%%%%
\begin{definition}[Weak-DKP condition]\label{wdkp.def}
Let $A : \ree_{+} \to \re^{n \times n}$ and $B : \ree_{+} \to \re^{n}$ be locally integrable functions.
Let $C \in \{W,R\}$. Define for $z \in \ree_{+}$ the quantities
\begin{align*}
\alpha_{A}^{C}(z) &\coloneqq \left( \fint_{C(z)} |A(w) - \la{A}_{C(z)} |^2 \, dw \right)^{1/2},\\
\alpha_{B}^{C}(z) &\coloneqq \left( \fint_{C(z)} |B(w)|^2 w_{n+1}^{2} \, dw \right)^{1/2},\\
\alpha_{A,B}^{C}(z) &\coloneqq \alpha_{A}^{C}(z) + \alpha_{B}^{C}(z),
\end{align*}
and the associated measures
\begin{align*}
\nu_{A,B}^{C} \coloneqq \bigl(\alpha_{A,B}^{C}(z)\bigr)^2 \frac{dz}{z_{n+1}},\quad  \nu_{A}^{C} \coloneqq \bigl(\alpha_{A}^{C}(z)\bigr)^2 \frac{dz}{z_{n+1}}\quad \text{ and } \quad \nu_{B}^{C} \coloneqq \bigl(\alpha_{B}^{C}(z)\bigr)^2 \frac{dz}{z_{n+1}}.
\end{align*}
We say that $(A,B)$ satisfies a weak DKP-condition in a set $E \subset \ree_{+}$ if  
\[
\no{\nu_{A,B}^{W}}_{\mathcal{C}(E)} < \infty .
\]
We say that $(A,B)$ satisfies a weak DKP-condition if  
it satisfies a weak DKP condition with $E= \ree_{+}$.
\end{definition}
%%%%%%%%%%%%%%%%%%%%%%%%%%%%%%%%%%%%%%%%%%%%%%%%%%%%%%%%%%%%%%%%
The following proposition shows that the Carleson condition above improves from Whitney to Carleson regions. Hence, the two conditions are in fact equivalent and we will drop the index $C$ from the notation. 

\begin{proposition}
\label{whitney-carleson.prop}
Let $M_0$ and $\varepsilon_0$ be given. Then there exists a constant $C$ such that the following holds.
Let $A$ be $M_0$-elliptic, $B$ be an $\varepsilon_0$-small drift and let $Q^n$ be a boundary cylinder.
Then 
\[
\no{\alpha_{A,B}^{R}}_{L^\infty(R(Q^n))}^2 + \no{\nu_{A,B}^{R}}_{\mathcal{C}(Q^n)} \le C \no{\nu_{A,B}^{W}}_{\mathcal{C}(9Q^n)}.
\] 
\end{proposition}
%%%%%%%%%%%%%%%%%%%%%%%%%%%%%%%%%%%%%%%%%%%%%%%%%%%%%%%%%%%%%%%%
\begin{proof}
The estimate for the $\alpha_A$ follows by the very same argument that proves \cite[Lemma~3.16 \& Remark~3.22]{DLM} in the elliptic setting and is a general fact about Carleson measures.  The estimate for $\alpha_B$ is even easier, since this quantity is an average of a function that is independent of the external variable $z$. Indeed, introducing additional averages over Whitney regions $\widetilde{W}(t,x,\lambda) \coloneqq Q^n(t,x,\lambda) \times (\tfrac{\lambda}{2},\lambda)$, we conclude by Fubini's theorem that
\begin{align}\label{whitney-carleson.prop.eq1}
\begin{split}
	\fint_{R(2Q^n)} |B(w)|^2 w_{n+1}^2 \, dw
	&= \fint_{R(2Q^n)} |B(w)|^2 w_{n+1}^2  \biggl( \fint_{\widetilde{W}(y)} dz \biggr) \, dw \\
	&\leq C \fint_{R(6Q^n)} \biggl(\fint_{W(z)} |B(w)|^2 w_{n+1}^2\, dw \biggr) \, dz.
\end{split}
\end{align}
The right-hand side is bounded from above by $C \|\nu_{B}^W\|_{\mathcal{C}(6Q^n)}$ and the left-hand side is bounded from below by $C \alpha_{B}^R(z)^2$ for every $z \in R(Q^n)$. Thus, the $L^\infty$-bound for $\alpha_B^R$ follows. By a similar use of Fubini's theorem we obtain
\begin{align*}
	\int_{R(Q^n)} \bigl(\alpha_B^R(z)\bigr)^2 \, \frac{dz}{z_{n+1}} 
	&\leq C \int_{R(3Q^n)} |B(w)|^2 w_{n+1}^2 \biggl(\int_{w_{n+1}/2}^{\infty} \, \frac{d z_{n+1}}{z_{n+1}^{2}} \biggr) \, dw \\
	&= C  \int_{R(3Q^n)} |B(w)|^2 w_{n+1} \, d w
\intertext{and repeating the argument in \eqref{whitney-carleson.prop.eq1} for this integral, we find}
	&\leq C \int_{R(9Q^n)} \biggl(\fint_{W(z)} |B(w)|^2 w_{n+1} \, dw \biggr) \, dz \\
	&\leq C \int_{R(9Q^n)} \bigl(\alpha_B^W(z)\bigr)^2 \, \frac{dz}{z_{n+1}}.
\end{align*}
This estimate is valid for every boundary cube, so $\no{\nu_{B}^{R}}_{\mathcal{C}(Q^n)} \le C \no{\nu_{B}^{W}}_{\mathcal{C}(9Q^n)}$.
\end{proof}
%%%%%%%%%%%%%%%%%%%%%%%%%%%%%%%%%%%%%%%%%%%%%%%%%%%%%%%%%%%%%%%%
Finally, we reserve the following notation for the $L^{2}$-norms of gradients, their linear approximations and their ratio.

\begin{definition}[Energies $J$ and $E$]
\label{JandEdefs.def}
For measurable functions $u$ on $\ree_+$ with locally integrable gradient we define 
\begin{align*}
J_{u}(t,x,\lambda) &\coloneqq \fint_{R(t,x,\lambda)} \bigl|\nabla (u -  z_{n+1} \la{\partial_{n+1}u}_{R(t,x,\lambda)} )\bigr|^{2} \, dz ,\\
E_u(t,x,\lambda) &\coloneqq \fint_{R(t,x,\lambda)} \bigl|\nabla u(z)\bigr|^{2} \, dz,\\
\beta_u(t,x,\lambda) &\coloneqq \left(\frac{J_{u} (t,x,r)}{E_{u} (t,x,r)}\right)^{1/2}. 
\end{align*}
\end{definition}
%%%%%%%%%%%%%%%%%%%%%%%%%%%%%%%%%%%%%%%%%%%%%%%%%%%%%%%%%%%%%%%%
% Note that if $u$ can be extended to a function that vanishes continuously on $Q$ then $\tilde\beta_u(t,x,r)\le \beta_u(t,x,r)$ by a Poincar\'e inequality. We will not use $\tilde\beta_u(t,x,r)$ explicitly below, but we will (implicitly) use this fact.

% {\Rd For the purposes of later discussion (extensions of the results here) we include the following definition. \simon{We also mention this when we talk about the Green function with drifts.}
% \begin{definition}[Strong Harnack Inequality]\label{strongharndef.def}
% Suppose $u \ge 0$ is a weak solution to $Lu = 0$ or $L^*u = 0$ on $\cR_{2Q}$ for a $n$-dimensional surface cube $Q = Q((s,y),r)$ and that $u$ vanishes continuously on $2Q$. We say that $u$ satisfies the strong Harnack inequality for all $(z,\tau,0) \in Q((s,y),r)$ and $r' \in (0,r/2]$ if there exists a constant $C > 1$ such that 
% \[C^{-1} u(\A_r^+(z,\tau)) \le  u(\A_r^-(z,\tau))  \le C u(\A_r^+(z,\tau))\]
% for all $(z,\tau,0) \in Q((s,y),r)$ and $r' \in (0,r/2]$.
% \end{definition}
% }
%%%%%%%%%%%%%%%%%%%%%%%%%%%%%%%%%%%%%%%%%%%%%%%%%%%%%%%%%%%%%%%%
\subsection{Basic estimates for weak solutions}
Next, we recall some basic estimates for solutions to parabolic equations. These estimates assume varying simplifications in the structure of the equation,
and will be needed in respective generality. 

The following Caccioppoli inequality is a special case of  Lemma~1 on p.~623 in \cite{MR0435594}. Indeed, by the observation in Section~\ref{subsec.weak.sol}, the singular drift can first be included in the source term $F$ for \cite{MR0435594} and then be absorbed to the left-hand side if $\varepsilon_0$ is small enough.
%%%%%%%%%%%%%%%%%%%%%%%%%%%%%%%%%%%%%%%%%%%%%%%%%%%%%%%%%%%%%%%%
\begin{lemma}[Caccioppoli inequality]
\label{caccioppoli.lem}
Let $M_0$ be given. Then there exists $\varepsilon_0, C > 0$ such that for any $c \in (1,2]$ the following holds. 
Let $Q = Q(z,r) \subset \ree_{+}$ be a parabolic cylinder, let $\cL$ be $(M_0,\varepsilon)$-parabolic, $F \in L^{2}_{loc}(\overline{\ree_{+}}; \mathbb{R}^{n})$ and $u$ be a solution to
\begin{align*}
  \cL u  = \div F \quad \text{or} \quad \cL^{*}u  = \div F \quad \text{in } c^2 Q \cap \ree_{+}.
\end{align*}
Then
\begin{multline*}
	\sup_{\tau} \fint_{Q^{\tau}  } \frac{|u(w)|^{2}}{r^{2}} \, dw + \fint_{Q } |\nabla u(w)|^{2} \,dw 
	\\
	\le \frac{C}{(c-1)^{2}} \fint_{cQ \cap \ree_{+} } \frac{|u(w)|^{2}}{r^{2}} \, dw + C \fint_{cQ  \cap \ree_{+} } |F(w)|^{2} \, dw.
\end{multline*}
in either of the following two scenarios:
\begin{enumerate}
  \item Interior case \& no drift: $cQ \subset \ree_{+}$ and $\varepsilon = 0$ .
  \item Boundary case \& Dirichlet conditions: $z \in \partial \ree_{+}$ and $u=0$ on $cQ  \cap \partial \ree_{+}$.
\end{enumerate}
\end{lemma}
%%%%%%%%%%%%%%%%%%%%%%%%%%%%%%%%%%%%%%%%%%%%%%%%%%%%%%%%%%%%%%%%
As a consequence, we get:

\begin{lemma}[Boundary Reverse H\"older inequality] \label{bdry-rhi-solution.lem}
Let $M_0$ be given. There exists $\varepsilon_0 , C> 0$ such that for any $c \in (1,2]$ the following holds. 
Whenever $\cL$ is $(M_0,\varepsilon_0)$-parabolic, $R = R(t,x,\lambda) \subset \ree_{+}$ is a Carleson region,
$F \in L^{2}_{loc}(\overline{\ree_{+}};\mathbb{R}^{n})$ and $u$ is a solution to
\begin{align*}
	\cL u  = \div F \quad \text{or} \quad \cL^{*}u  = \div F \quad 
\end{align*}
in $c^2 R$ such that $u=0$ on $Q^{n}(cR)$, then for $q = \frac{2(n+2)}{n}$ it follows that
\begin{multline*}
\left( \fint_{R} |u(z)|^{q} \,dz  \right)^{1/q}
  \leq \frac{C}{(c-1)} \left( \fint_{cR} |u(z)|^{2} \, dz \right)^{1/2} 
+ C \lambda  \left( \fint_{cR} |F(z)|^{2} \, dz \right)^{1/2} .
\end{multline*}
\end{lemma} 
%%%%%%%%%%%%%%%%%%%%%%%%%%%%%%%%%%%%%%%%%%%%%%%%%%%%%%%%%%%%%%%%
\begin{proof}
Pick any $p \in (q, 2n/(n-2))$ and fix $\theta \in (0,1)$ with $1/q = (1-\theta)/2 + \theta/p$.
Interpolating the slice-wise $L^{q}$-norm, we obtain 
\begin{align} \label{bdry-rhi-solution.lem.eq1}
\begin{split}
\left(\fint_{R} |u(z)|^{q} \,dz \right)^{1/q}
&\le  \left(\sup_{\tau} \fint_{R^\tau} |u(z)|^{2} \, dz \right)^{(1-\theta)/2} \\
 &\quad \times \biggl( \frac{1}{\lambda^2} \int_{0}^{\lambda^2} \biggl(\fint_{R^\tau} |u(z)|^{p} \, dz  \biggr)^{\theta q/p} \, dt \biggr)^{1/q}.
\end{split}
\end{align} 
A slice-wise Sobolev--Poincar\'e inequality yields,
\begin{align*}
	 \int_{0}^{\lambda^2} \left(\fint_{R^\tau} |u(z)|^{p} \, dz  \right)^{\theta q/p} \, dt 
&\le C \int_{0}^{\lambda^2} \left(\fint_{R^\tau} |\lambda \nabla u(z)|^{2} \, dz  \right)^{\theta q/2} \, dt \\
&\le C \lambda^{2+ \theta q/2} \left(  \fint_{R} |\nabla u(z)|^{2} \, dz    \right)^{\theta q/2},
\end{align*}
where in the second step we have used $\theta q/2 = (2/n) \cdot (1-2/p)^{-1} \geq 1$ and Jensen's inequality. Now, both factors on the left of \eqref{bdry-rhi-solution.lem.eq1} can be controlled by the Caccioppoli inequality (Lemma~\ref{caccioppoli.lem}) and the claim follows.
\end{proof}
%%%%%%%%%%%%%%%%%%%%%%%%%%%%%%%%%%%%%%%%%%%%%%%%%%%%%%%%%%%%%%%%
The following boundary estimates on derivatives to drift free constant coefficient equations are folklore but difficult to trace down in the literature. We include a proof for convenience.

\begin{lemma}\label{constderests.lem}
Let $M_0$ and $c>1$ be given. For any $k \in \N_0, \gamma \in \N_0^n$ there exists a constant $C$ such that the following holds.
If $\cL_0$ is a constant coefficient $(M_0,0)$-parabolic operator, $R = R(t,x,\lambda) \subset \ree_{+}$ is a Carleson region and $u$ is a global solution to either $\cL_0 u = 0$ or $\cL_0^{*}u= 0$
in $cR$ with $u=0$ on $Q^{n}(cR)$, then
\begin{equation*}
%\label{constestlem.eq}
\sup_{z \in R} |\partial_t^{k} \partial^\gamma u(z)| \le \frac{C}{\lambda^{2k+|\gamma|-1}} \left(\fint_{cR} |\nabla u (z)|^{2} \, dz \right)^{1/2} .
\end{equation*}
Moreover, $u$ vanishes continuously on $Q^{n}(cR)$.
\end{lemma}
%%%%%%%%%%%%%%%%%%%%%%%%%%%%%%%%%%%%%%%%%%%%%%%%%%%%%%%%%%%%%%%%
\begin{proof} 
We can assume $(t,x, \lambda) = (0,0,1)$ since the general case can be obtained by re-scaling. As the coefficients of $\cL_0$ are constant, it follows from the method of difference quotients and the Caccioppoli inequality on interior cubes that all derivatives $\partial_t^k \partial^\gamma u$ are locally square-integrable functions that solve the same equation.

In order to prove the boundary estimates, we first claim that 
\begin{align} \label{constderests.lem.eq1}
	\| \nabla \partial_{x_j} u\|_{L^2(R)} + \|\partial_t u\|_{L^2(R)}+ \| \nabla \partial_t u \|_{L^2(R)} \leq C \| \nabla u\|_{L^2(cR)}
\end{align}
for all $j=1,\ldots,n-1$. The estimate for $\nabla \partial_{x_j} u$ follows by applying the boundary Caccioppoli inequality to (difference quotients approximating) $\partial_{x_j}u$. In order to estimate $\partial_t u$, we let $\chi$ be a smooth function with $\chi = 1$ on  $R$, $\chi = 0$ on $\ree_{+} \setminus c R$ and $ \no{\partial_t \chi}_{\infty} + \no{\nabla \chi}_{\infty} \le C$ for some constant $C=C(n,c)$. The localized solution $\tilde{u} \coloneqq u \chi$ vanishes on the entire parabolic boundary of $c R$ and satisfies $\cL_0 \tilde{u} = f + \div F$ with 
\begin{align}
	\label{RHS.localized.eq}
	\begin{split}
		f &\coloneqq u \chi + (\partial_t \chi) u - A\nabla u \cdot \nabla \chi, \\
		F &\coloneqq - A(u\nabla \chi).
	\end{split}
\end{align}
Since $A$ is constant, we have
\begin{align*}
	\|f + \div F\|_{L^2(cR)} 
	\leq C \bigl(\|u\|_{L^2(c R)} + \|\nabla u\|_{L^2(c R)}\bigr)
	\leq C \|\nabla u\|_{L^2(c R)},
\end{align*}
where we have used the Poincar\'e inequality in the second step. Classical higher regularity for second-order parabolic equations as in \cite[Chapter 7, Theorem~5]{Evans} yields $\|\partial_t \tilde{u}\|_{L^2(R)} \leq C \|\nabla u\|_{L^2(cR)}$ and hence the bound for $\partial_t u$ in \eqref{constderests.lem.eq1}. The bound for $\nabla \partial_t u$ then follows by the Caccioppoli  inequality applied to $\partial_t u$. 

Now, let $k \in \N_0$ and $\gamma \in \N_0^n$.  By the equation, we may write $\partial_t^{k}\partial^\gamma u$ as  a linear combination of terms $\partial_t^{k'} \partial^{\gamma'} u$ such that the order of differentiation is one to the vertical $\lambda$-direction. Since \eqref{constderests.lem.eq1} applies to the $x$- and $t$-derivatives of $u$, we find $\|\partial_t^{k}\partial^\gamma u\|_{L^2(R)} \leq C  \|\nabla u\|_{L^2(c R)}$.
Since this bound holds for all $k,\gamma$, the analogous bound with $L^\infty$-norm on the left follows by Sobolev embeddings and so does the claim that $u$ vanishes continuously on $Q^{n}(cR)$ .
\end{proof}
%%%%%%%%%%%%%%%%%%%%%%%%%%%%%%%%%%%%%%%%%%%%%%%%%%%%%%%%%%%%%%%%
Finally, we need a reverse H\"older estimate for the gradient of adjoint solutions near the boundary. A Gehring-type argument is not applicable here, as we do not anticipate such an estimate in the interior. Instead, we utilize an analytic perturbation argument similar to \cite{ABES_JMPA} in the Banach spaces
\begin{align}\label{EpSpaces.eq}
\begin{split}
	E_p &\coloneqq L^{p}(\R ; W_0^{1,p}(\R_{+}^{n})) \cap H^{1/2,p}(\R ; L^{p}(\R_{+}^{n})), \\
	\|\bigdot\|_{E_p}^p &\coloneqq \|\bigdot\|_p^p + \|\nabla \bigdot\|_p^p + \|D_t^{1/2} \bigdot\|_p^p,
\end{split}
\end{align}
where $p \in (1,\infty)$, $t \in \R$ is the distinguished variable and the half-order time derivative $D_t^{1/2}$ is defined via the Fourier multiplier $\tau \mapsto |\tau|^{1/2}$ in the $t$-variable. Further background on the role of these spaces in parabolic PDEs can be found in \cite[Sect.~7]{ABES_JMPA}. For completeness, we provide a proof of the following complex interpolation result at the end of the section.
%%%%%%%%%%%%%%%%%%%%%%%%%%%%%%%%%%%%%%%%%%%%%%%%%%%%%%%%%%%%%%%%
\begin{lemma}\label{EpInterpol.lem}
The spaces $E_p$ in \eqref{EpSpaces.eq} interpolate by the complex method according to the rule
\begin{align*}
	[E_{p_0}, E_{p_1}]_\theta = E_p, \qquad \theta \in (0,1), \quad (1-\theta)/p_0 + \theta/p_1 = 1/p
\end{align*}
and the same result holds for the (anti-)duals $E_p^*$.
\end{lemma}
%%%%%%%%%%%%%%%%%%%%%%%%%%%%%%%%%%%%%%%%%%%%%%%%%%%%%%%%%%%%%%%%
\begin{lemma}[Boundary Reverse H\"older inequality for the gradient] \label{RHforgflem.lem}
Let $M_0$ and $c>1$ be given.
Then there exist $\varepsilon_0, C> 0$ and $p>2$ such that the following holds.
Let $R =R(t,x,\lambda) \subset \ree_{+} $ be a Carleson box and let $\cL$ be an $(M_0,\varepsilon_0)$-parabolic operator.
Whenever $u$ is a solution to $\cL^{*}u = 0$ on $c R$ with $u= 0$ on $Q^{n}(cR)$, then
\[
\left( \fint_{R} |\nabla u(z)|^{p} \, dz \right)^{1/p} 
\le C \left( \fint_{cR} |\nabla u(z)|^{2} \, dz \right)^{1/2}.
\]
\end{lemma}
%%%%%%%%%%%%%%%%%%%%%%%%%%%%%%%%%%%%%%%%%%%%%%%%%%%%%%%%%%%%%%%%
\begin{proof} 
By scaling, we may assume $\lambda = 1$. Let $\chi$ be a smooth function with $\chi = 1$ on  $R$, $\chi = 0$ on $\ree_{+} \setminus c^{1/2} R$ and $ \no{\partial_t \chi}_{\infty} + \no{\nabla \chi}_{\infty} \le C$ for some dimensional constant $C$, and set 
\[
L_1 u \coloneqq \partial_t u + \div A \nabla u, \quad  L_2u \coloneqq \div (Bu), \quad Lu \coloneqq u + L_1 u + L_2u.
\]
Then $\tilde{u} \coloneqq u\chi \in L^2(\R; W^{1,2}_0(\R_{+}^n))$  satisfies a global equation 
\[
L \tilde{u} = f + \div F
\]
on the upper-half-space with right-hand side
\begin{align*}
	%\label{RHS.Sneiberg.eq}
	\begin{split}
		f &\coloneqq u \chi + (\partial_t \chi) u - A\nabla u \cdot \nabla \chi - Bu \cdot \nabla \chi, \\
		F &\coloneqq - A(u\nabla \chi).
	\end{split}
\end{align*}
By the observation in Section~\ref{subsec.weak.sol} we have $Bu \in L_{loc}^2(\overline{\ree_{+}})$ and therefore the equation can be rewritten as $\partial_t \tilde{u} = \div(G) + h$ with right-hand sides $G, h \in L^2(\R^{n+1}_+)$. This means that $\partial_t \tilde{u} \in L^2(\R; W^{1,2}_0(\R_{+}^n)^*)$ and, using Plancherel's theorem for the Fourier transform in the $t$-variable, we conclude $u \in E_2$.

\emph{Step 1: Variational Estimates}.
We are going to use the full scale of spaces $E_p$ and define the operator $L: E_p \to E_{p'}^*$ variationally by
\begin{align*}
	\la{Lu, v} = \int_{\R^{1+n}_+} u \cdot \overline{v} + H_tD_t^{1/2}u \cdot \overline{D_t^{1/2}v} + A \nabla u \cdot \overline{\nabla v} - Bu \cdot \overline{\nabla v} \, dz,
\end{align*}
where $H_t$ is the Hilbert transform in the $t$-variable. In order to see that $L$ is well-defined, we use H\"older's inequality to bound
\begin{align*}
	|\langle L_1 v, w \rangle|
	&\leq \int_{\R^{1+n}_+} |H_tD_t^{1/2}v| |D_t^{1/2}w| + M_0| \nabla v |\nabla w|  \, dz \\
	&\leq  \|H_tD_t^{1/2}v\|_p  \|D_t^{1/2}w\|_{p'} + M_0 \|\nabla v\|_p \|\nabla w\|_{p'},
\end{align*}
and also invoke the one-dimensional Hardy inequality to get
\begin{align}
	\label{Hardy.eq}
	|\langle  L_2v, w \rangle|
	\leq \int_{\R^{1+n}_+} \varepsilon_0 \frac{|v|}{z_{n+1}} |\nabla w| \, dz
	\leq  \frac{\eps_0p}{p-1}  \|\partial_{n+1} v \|_{p}  \|\nabla w\|_{p'}.
\end{align}
Altogether, $L: E_p \to E_{p'}^*$  with norm depending on $M_0, \eps_0, p$. 

If $p=2$, then we have the `hidden coercivity bound'
\begin{align*}
	\Re \la{(1+L_1)v, (1+ \delta H_t)v} \geq (M_0^{-1} - \delta M_0) \|\nabla v\|_2^2 + \delta  \|D_t^{1/2}v\|_{2}^2 + \|v\big\|_2^2,
\end{align*}
compare \cite[Lemma~2.3]{ABES_JMPA}. We take $\delta \coloneqq M_0^{-1}(M_0+1)^{-1}<1$ to have equal constants for the first terms on the right-hand side and then require $\eps_0 < \delta/4$ in order to obtain from \eqref{Hardy.eq} a similar estimate for the full operator,
\begin{align*}
	\Re \la{L v, (1+ \delta H_t)v} \geq \frac{\delta}{2} \|v\|_{E_2}^2.
\end{align*}
This implies that  $L: E_2 \to E_2^*$ is invertible, see again \cite[Lemma~2.3]{ABES_JMPA}. 

Since the spaces $E_p$ and their duals interpolate by the complex method,  we can apply  Sneiberg's lemma (\cite[Theorem~A1]{MR3907738}, \cite{MR0634681}) to conclude that $L$ remains invertible as an operator $E_p \to E_{p'}^{*}$ and that the inverses agree on common subspaces, whenever $2 \leq p <p_0$, where  $p_0$ and the norm of the inverse depend only on the dimension and $M_0$.

\emph{Step 2:Conclusion}. If necessary, we lower $p_0$ to achieve $p_0 < \min(2(n+2)/n,n)$ and lower $\varepsilon_0$ to have Lemma~\ref{bdry-rhi-solution.lem} at our disposal.

Let $2 < p < p_0$. By H\"older's inequality, $\|\div F\|_{E_{p'}^*} \leq \|F\|_p$ and  the Sobolev embedding for $\smash{W^{1,p}_0(\R_{+}^n)}$,  also $\|f\|_{E_{p'}^*} \leq \|f\|_{p_*}$ with exponent $1/p_* = 1/p - 1/n$. Thus, 
\begin{align*}
	\|\div F \|_{E_{p'}^*} \leq  C \|u \|_{L^p(c^{1/2} R)} \leq C \|u \|_{L^2(cR)} \leq C \|\nabla u \|_{L^2(cR)},
\end{align*}
where we use Lemma~\ref{bdry-rhi-solution.lem}  in the third step and the Poincar\'e inequality with zero boundary values on one side of a cube in the fourth step. Likewise, in estimating
\begin{align*}
	\| f\|_{E_{p'}^*} 
	\leq  C \bigl(\|u \|_{L^{p_*}(c^{1/2} R)} +C \|\nabla u\|_{L^{p_*}(c^{1/2} R)}  \bigr)  \leq C \|\nabla u \|_{L^{p_*}(c^{1/2} R)} 
	\leq C \|\nabla u \|_{L^2(c^{1/2} R)},
\end{align*}
we use the one-dimensional Hardy inequality, the same Poincar\'e inequality and finally that $p_* \leq 2$. Altogether, $u \in E_2$ satisfies $L \tilde{u} = \div F + f\in E_{p'}^*$. Invertibility shown in Step~1 yields the claim
\begin{align*}
	\|\nabla u \|_{L^p(R)}
	\leq \|\tilde{u} \|_{E_p}
	\leq \ \|L\tilde{u} \|_{E_{p'}^*}
	\leq C \|\nabla u \|_{L^2(cR)}. &\qedhere
\end{align*}
\end{proof}
%%%%%%%%%%%%%%%%%%%%%%%%%%%%%%%%%%%%%%%%%%%%%%%%%%%%%%%%%%%%%%%%
\begin{proof}[Proof of Lemma~\ref{EpInterpol.lem}]
Let us call
\begin{align*}
	E_{p}^{full} \coloneqq L^{p}(\R ; W ^{1,p}(\R^{n})) \cap H^{1/2,p}(\R ; L^{p}(\R^{n}))
\end{align*}
the corresponding spaces on the full space $\R^{1+n}$. By \cite[Lemma~6.1]{MR4127944} they are isomorphic to $L^p(\R; L^p(\R^n))$ and follow the interpolation rules in question. 

Extending by zero to the lower half-space, the space $E_p$ is hence isometrically identified as a closed subspace of $E_p^{full}$ that we also call $E_p$. Denoting by $\pi_1$ the restriction $L^{p}(\R ; W ^{1,p}(\R^{n+1})) \to L^{p}(\R ; W ^{1,p}(\R_{-}^{n+1}))$ and by $\pi_2: L^{p}(\R ; W ^{1,p}(\R_{-}^{n+1})) \to L^{p}(\R ; W ^{1,p}(\R^{n+1}))$ the extension by even reflection, we obtain a bounded projection $\pi \coloneqq \pi_2 \circ \pi_1$ of $\smash{E_p^{full}}$ with kernel $E_p$. The interpolation principle for complemented subspaces \cite[Sect.~1.17.1]{MR0503903} tells us that $E_p$ and $E_p^{full}$ interpolate according to the same rules. 

The construction above reveals that $E_p$-spaces are isomorphic to closed subspaces of reflexive spaces, hence reflexive. Moreover, $C_c^\infty(\R_{+}^{1+n})$ is dense in all $E_p$ by a standard smoothing procedure. Consequently, the interpolation rules for $E_p^*$ follow by a duality principle for the complex method \cite[Cor.~4.5.2]{MR482275}.
\end{proof}
%%%%%%%%%%%%%%%%%%%%%%%%%%%%%%%%%%%%%%%%%%%%%%%%%%%%%%%%%%%%%%%%
\subsection{Pointwise estimates for weak solutions}
Local boundedness and (H\"older-) continuity of solutions to general $(M_0,\eps_0)$-parabolic equations and their adjoints is known since the work of Aronson~\cite{MR0435594}. We need the following pointwise estimates. 
%%%%%%%%%%%%%%%%%%%%%%%%%%%%%%%%%%%%%%%%%%%%%%%%%%%%%%%%%%%%%%%%
\begin{lemma}[Parabolic Harnack's inequality {\cite[Theorem~2]{MR159139}}, {\cite{MR288405}}]
\label{parabolic-harnack.lem}
Let $M_0$ be given. There is a constant $C$ such that the following holds. Let $Q$ be a parabolic cylinder with $2Q \subset \ree_{+}$ and let $\cL$ be $(M_0,0)$-parabolic.
Then
\[
u(t,x,\lambda) \le u(s,y,\mu) \exp\left( \frac{C |(y,\mu)-(x,\lambda)|^{2}}{|t-s|} + 1   \right)
\]
for all $(t,x,\lambda), (s,y,\mu) \in Q$ in either of the following scenarios:
\begin{enumerate}
  \item $s > t$ and $u \ge 0$ solves $\cL u=0$ in $2Q$.
  \item $t > s$ and $u \ge 0$ solves $\cL^{*}u=0$ in $2Q$.
\end{enumerate}
\end{lemma}
%%%%%%%%%%%%%%%%%%%%%%%%%%%%%%%%%%%%%%%%%%%%%%%%%%%%%%%%%%%%%%%%
\begin{lemma}[Boundary Harnack's inequality \cite{Fab-Saf} {\cite[Theorem~1.6]{FGS}}]
\label{bdryharnack.lem}
Let $M_0$ be given.
There is a constant $C$ such that whenever $R = R(t,x,\lambda)$ is a parabolic Carleson region, $\cL$ is $(M_0,0)$-parabolic and $u,v \ge 0$ are weak solutions to $\cL u = 0$ on $2R$ that vanish continuously on the boundary cylinder $Q^n(2R)$, then
\[
\frac{u(z)}{v(z)} \le C \frac{u(a^+(t,x,\lambda) )}{v(a^-(t,x,\lambda) )} \qquad (z \in R).
\]
\end{lemma}
%%%%%%%%%%%%%%%%%%%%%%%%%%%%%%%%%%%%%%%%%%%%%%%%%%%%%%%%%%%%%%%%
\begin{lemma}[Carleson Estimate {\cite[Theorem~3.1]{MR640782}}] \label{carlest.lem}
Let $M_0$ be given.
There is a constant $C$ such that whenever $R = R(t,x,\lambda)$ is a parabolic Carleson region, $\cL$ is $(M_0,0)$-parabolic and $u \ge 0$ is a weak solution to $\cL u = 0$ on $2R$ that vanishes continuously on the boundary cylinder $Q^n(2R)$, then
\[
u(z) \le C u(a^+(R)) \qquad (z \in \tfrac{1}{2} R).
\]
\end{lemma}
%Note: Corkscrew point place between 2R (where we solve) but outside of box where we want to estimate. This is because Carleson box R goes up to 2lambda.
%%%%%%%%%%%%%%%%%%%%%%%%%%%%%%%%%%%%%%%%%%%%%%%%%%%%%%%%%%%%%%%%
By a simple reversal of time, we see that in Lemmas~\ref{parabolic-harnack.lem} and \ref{bdryharnack.lem} for solutions to the adjoint equation $\cL^* u= 0$, the roles of forward and backward time-lag are interchanged. Combining {\cite[Theorem~6.32]{GARY}}, the Carleson Estimate and the Parabolic Harnack's inequality also yields the following useful bound.

\begin{lemma}[Boundary H\"older Continuity]\label{bhc.lem}
Let $M_0$ be given. There are constants $C$ and $\alpha \in (0,1)$ such that whenever $R=R(t,x,\lambda)$ is a parabolic Carleson region, $\cL$ is $(M_0,0)$-parabolic and $u \geq 0$ is a weak solution to $\cL u = 0$ or $\cL^*u = 0$ on $2R$ that vanishes continuously on the boundary cylinder $Q^n(2R)$, then
\[
u(z) \le C \left(\frac{z_{n+1}}{\lambda}\right)^\alpha \left(\fint_{2R} |u(w)|^2 \, dw \right)^{1/2} \qquad (z  \in \tfrac{1}{2} R).
\]
\end{lemma}
%\begin{proof}
%This follows from \cite[Lemma~3.9]{HL-Mem} and the parabolic Harnack's inequality Lemma~\ref{parabolic-harnack.lem}.
%\end{proof}
%%%%%%%%%%%%%%%%%%%%%%%%%%%%%%%%%%%%%%%%%%%%%%%%%%%%%%%%%%%%%%%%
\subsection{Green functions for equations with drift}

In this section, we provide the definition and estimates for the Green function and the parabolic measure of an $(M_0,\varepsilon_0)$-parabolic operator $\cL$. 

The continuous Dirichlet problem for $\cL$ consists in finding for given $f \in C_c(\rn)$ a weak solution to $\cL u = 0$ in $\ree_+$ that is continuous up to the boundary and attains the boundary value $f$. If this problem is uniquely solvable for all $f$, then the parabolic measure $\omega^z$ at $z \in \ree_{+}$ is defined as the measure for which 
\[
u(z) = \int_{\partial \ree_{+}} f(w) \, d \omega^z(w),
\]
whenever $u$ and $f$ are related as above.
By a Green's function $G : \ree_{+} \times \ree_{+} \to [0,\infty)$ we mean a function satisfying the following properties.
\begin{enumerate}
  \item $G(\bigdot,z)$ with $z \in \ree_{+}$ fixed is a weak solution to $\cL u = 0$ in $\ree_{+} \setminus \{z\}$.
  \item $G(z,\bigdot)$ with $z \in \ree_{+}$ fixed is a weak solution to $\cL^{*}u = 0$ on $\ree_{+} \setminus \{z\}$.
  \item If $z, w \in \ree_{+}$ and $w_1 > z_1$, then $G(z,w) = 0$.
  \item If $z \in \ree_{+}$, then both $G(\bigdot, z)$ and $G(z,\bigdot)$ extend continuously to $\overline{\ree_{+}}$ and vanish at $\partial \ree_{+}$.
  \item If $\Psi \in C_c^\infty(\ree)$, then 
\begin{align}
\label{Rieszformeq.eq}
    \Psi(z) &= \la{\cL^{*}G(z,\bigdot) , \Psi}  + \int_{\partial \ree_{+}} \Psi(w) \, d\omega^{z}(w) \qquad (z \in \ree_{+}).
\end{align}
\end{enumerate}
The identity \eqref{Rieszformeq.eq} is called the Riesz formula (for $\mathcal{L}$). 

For the equations we consider here, these objects exist. For readers that are perhaps worried about smoothness assumptions (on the drift) in \cite{HL-Mem}, we remark that \cite{HL-Mem} constructs these objects in two settings: drift-free equations and a perturbative regime, where the `unperturbed'  operator has parabolic measure in (weak-)$A_\infty$\cite[Chapter~III]{HL-Mem}. Therefore one can run our arguments here first without the drift and prove that the associated parabolic measure is in $A_\infty$. Then the perturbative regime covers our $\eps_0$-small drifts with a Carleson condition. For instance, in Chapter~III of \cite{HL-Mem}, Theorem 1.7 implies (1.5)(b) for the operator $\cL$ (with drift) and therefore, by Lemma 2.2 and Lemma 2.6, the Green function and parabolic measure for $\cL$ exist. We state this observation explicitly as
%%%%%%%%%%%%%%%%%%%%%%%%%%%%%%%%%%%%%%%%%%%%%%%%%%%%%%%%%%%%%%%%
\begin{proposition}
Given $M_0$, there exists $\varepsilon_0 > 0$ such that  if $\cL$ is an $(M_0,\varepsilon_0)$-parabolic operator, 
then the continuous Dirichlet problem for $\cL$ is uniquely solvable and the parabolic measure and the parabolic Green function on $\ree_{+}$ exist.
\end{proposition}
%%%%%%%%%%%%%%%%%%%%%%%%%%%%%%%%%%%%%%%%%%%%%%%%%%%%%%%%%%%%%%%%
We recall the following set of useful estimates on parabolic measures and Green functions.

\begin{lemma}[Estimates for parabolic measure and Green function]\label{HL-Green-pm.lem}
Given $M_0$, there exists $\varepsilon_0 > 0$, $\kappa_0 \ge 40$ and $C \ge 1$ such that the following statement is valid. If $\kappa \ge \kappa_0$,  $(t_0,x_0,\lambda_0) \in \ree_{+} $ and $p  \coloneqq a^{+}(t_0,x_0,\kappa \lambda_{0})$, then for an $(M_0,\varepsilon_0)$-parabolic operator the Green function $G$ and the parabolic measure $\omega^p$ at $p$ satisfy: 
\begin{enumerate}
  \item \emph{Strong Harnack inequality}:
\[
\sup_{w \in W(z)} G(p,w) \le C \inf_{w \in W(z)} G(p,w) \qquad (z \in R(t_0,x_0, 4 \lambda_0)).
\]
  \item \emph{Local doubling for parabolic measure}:
\[
\omega^{p}(Q^n(t,x,2\lambda)) \le C \omega^{p}(Q^n(t,x,\lambda)) \qquad ((t,x,\lambda) \in R(t_0,x_0, 4 \lambda_0)).
\]
  \item \emph{CFMS-estimate}: 
 \[
\frac{\omega^{p}(Q(t,x,\lambda))}{\lambda^{n+1}} \approx_C \frac{G(p,(t,x,\lambda))}{\lambda} \qquad ((t,x,\lambda) \in R(t_0,x_0, 4 \lambda_0)).
 \]
  \item \emph{(Backward) Carleson estimate}: If $(t,x,\lambda) \in R(t_0,x_0, 4 \lambda_0)$, then
\[
G(p,(s,y,\mu)) \le C G(p,a^{+}(t,x,\lambda)) \qquad ((s,y,\mu) \in R(t,x,\tfrac{3}{4}\lambda)). 
\]
  \item \emph{Gradient estimate}: 
\[
\frac{G(p,(t,x,\lambda))}{\lambda} \approx_C \left(\fint_{R(t,x,\lambda)} |\nabla_zG(p,z)|^{2} \, dz \right)^{1/2} \qquad ((t,x,\lambda) \in R(t_0,x_0, 4 \lambda_0)),
\]
where $\nabla_z$ is the spatial gradient in the second set of variables for $G$.
\end{enumerate}
\end{lemma}
%%%%%%%%%%%%%%%%%%%%%%%%%%%%%%%%%%%%%%%%%%%%%%%%%%%%%%%%%%%%%%%%
\begin{proof}
The first four items are straightforward reformulations of Lemmata 3.9 to 3.14 in \cite{HL-Mem}. The uniformity of the estimates with respect to $\kappa$ and the possibility to take an explicit ambient box $R(t_0,x_0,4 \lambda_0)$ also follows from \cite{HL-Mem}, which states results for all $(t_0,x_0,\lambda_0)$, some $\kappa_0$ and some ambient box $R(t_0,x_0, c \lambda_0)$: indeed, $a^+(t_0, x_0, \kappa \lambda_0) = a^+(t_0, x_0, \kappa_0 \lambda_1)$ for $\lambda_1 \coloneqq \lambda_0 \kappa/\kappa_0$ and under this change this change of reference point,  $Q(t_0,x_0, (c \kappa/\kappa_0) \lambda_0 ) = Q(t_0,x_0, c \lambda_1)$. 

In one direction, we use the strong Harnack inequality to control the left-hand side in (5) by the Whitney average $(\fint_{W(t,x, \lambda)}| \frac{G(p,z)}{\lambda}|^2 \, dz)^{1/2}$, enlarge to a Carleson average and then apply the Poincar\'e inequality with vanishing boundary values at $z_{n+1}=0$. Conversely, we use  the boundary Caccioppoli inequality and the backward Carleson estimate to control the right-hand side in (5) by $\frac{G(p, a^{+}(t,x,2\lambda))}{\lambda}$ and then conclude by the strong Harnack inequality.
\end{proof}
%%%%%%%%%%%%%%%%%%%%%%%%%%%%%%%%%%%%%%%%%%%%%%%%%%%%%%%%%%%%%%%%
\section{Just estimates}
\noindent Here, we prove our key estimates on the energies introduced in Definition \ref{JandEdefs.def}. We split this part into four:
First constant coefficient operators, then error estimates for differences of solutions and finally we combine the two points of view to conclude oscillation estimates for Green functions and parabolic measure. 
%%%%%%%%%%%%%%%%%%%%%%%%%%%%%%%%%%%%%%%%%%%%%%%%%%%%%%%%%%%%%%%%
\subsection{Estimates for energies of constant coefficient operators}
\label{Estimates for constant coefficient operators}
We start with estimates for the energies from Definition~\ref{JandEdefs.def} in case of constant coefficient equations.

The decay estimate is a standard estimate on regularity of solutions.
An interior version can be found in many textbooks and the boundary version for elliptic equations is provided in \cite[Lemma~3.4]{DLM}. The parabolic version follows with no additional complications.
The non-degeneracy estimate is very different: it relies on the strong Harnack inequality and we do not expect a full analogy of the elliptic theory.
For our applications, however, it suffices to work with solutions that share many values with a parabolic Green function.
%%%%%%%%%%%%%%%%%%%%%%%%%%%%%%%%%%%%%%%%%%%%%%%%%%%%%%%%%%%%%%%%
\begin{lemma}\label{dec-lb-jande.lem}
Let $M_0$ be given.
There exists $\varepsilon_0 > 0$, $\kappa_0 \ge 40$ and $C \ge 1$ such that whenever $\cL_0$ is  a constant coefficient $(M_0,0)$-parabolic operator, $\kappa \ge \kappa_0$ and $(t_0,x_0, \lambda_0), (t,x,\lambda) \in \ree_{+}$, then the following holds.
\begin{enumerate}
  \item \emph{Decay of $J$}: Whenever $u_0$ solves $\cL_0^{*}u=0$ in $R(t,x,\lambda)$ and satisfies $u_0=0$ on $Q^{n}(t,x,\lambda)$, then 
\[
J_{u_0}(t,x,\lamprime) \le C \left( \frac{\lamprime}{\lambda} \right)^{2} J_{u_0}(t,x,\lambda) \qquad (0< \lamprime \leq \lambda).
\]
  \item \emph{Non-degeneracy of $E$}:  Let $G$ be the Green function for an $(M_0,\varepsilon_0)$-parabolic operator and set $u  \coloneqq G(a^+(t_0,x_0, \kappa \lambda_0), \bigdot)$. If $(t,x,\lambda) \in R(t_0,x_0,\lambda_0)$ and $u_0$ solves $ \cL_0^*u_0 = 0 $ in $R(t,x,2\lambda)$ with $u_0  = u$ on $\partial_{par^*}  R(t,x,2\lambda)$ (continuously on $Q^n(t,x,2\lambda)$), then
\[
E_{u_0}(t,x,\lambda) \le C  E_{u_0}(t,x,\lamprime) \qquad (0<\lamprime \leq \lambda).
\]
\end{enumerate} 
\end{lemma}
%%%%%%%%%%%%%%%%%%%%%%%%%%%%%%%%%%%%%%%%%%%%%%%%%%%%%%%%%%%%%%%%
\begin{proof}
We start with the first item. The estimate follows by the triangle inequality if $\lamprime' > \lambda/2$, so we may assume $\lamprime \leq \lambda/2$ from now on. In this case we bound
\begin{align*}
J_{u_0}(t,x,\lamprime) 
\leq \sup_{(s,y,\mu) \in R(t,x, \lamprime)} |\nabla_y u_0(s,y,\mu)|^{2}
 + \fint_{R(t,x, \lamprime)} |\partial_{\mu} u_0- \la{\partial_{\mu} u_0}_{R(t,x, \lamprime)} |^{2} \, dz
\coloneqq \I + \II.
\end{align*}
As $u_0$ and its lateral derivatives vanish on $Q^{n}(t,x,\lambda)$, the mean-value theorem yields
\begin{align*}
\I 
\leq \sup_{(s,y,\mu) \in R(t,x, \lamprime)} |\mu \partial_{\mu} \nabla_y u_0(s,y,\mu)|^{2} 
&\leq C \left( \frac{\lamprime}{\lambda} \right)^{2} \fint_{R(t,x, \lambda/2)} |\lambda \nabla(\nabla_y u_0)|^{2} \, d z \\
&\leq C \left( \frac{\lamprime}{\lambda} \right)^{2} J_{u_0}(t,x,\lambda),
\end{align*}
where in the second and third step we have used Lemma~\ref{constderests.lem} and the boundary Caccioppoli inequality for the $y$-derivatives of $u_0$. Next, we introduce $v \coloneqq u_0 - \mu \la{\partial_{\mu} u_0}_{R(t,x, \lamprime)}$ and note that this function has the same properties as $u_0$ since $\cL_0$ has constant coefficients. Poincar\'e's inequality yields
\begin{align*}
\II 
\leq C \fint_{R(t,x, \lamprime)} |\lamprime \nabla \partial_{\mu} u_0|^{2} \, dz 
&= C \fint_{R(t,x, \lamprime)} |\lamprime \nabla \partial_{\mu} v|^{2} \, dz \\
&\le C \sup_{(s,y,\mu) \in R(t,x, \lambda/2)} |\lamprime \nabla (\partial_{\mu} v)|^{2} \\
&\le C \left( \frac{\lamprime}{\lambda} \right)^{2} \fint_{R(t,x, \lambda)} |\nabla v|^2 \, d z
=C \left( \frac{\lamprime}{\lambda} \right)^{2} J_{u_0}(t,x,\lambda),
\end{align*} 
where we have used Lemma~\ref{constderests.lem} in the penultimate step. This completes the proof.

We turn to the second item. By the boundary Caccioppoli inequality (Lemma~\ref{caccioppoli.lem}) and the maximum principle (e.g.\ \cite[Corollary~6.26]{GARY}), 
\begin{align*}
E_{u_0}(t,x,\lambda) \leq \frac{C}{\lambda^2}  \fint_{R(t,x,2\lambda)} |u_0|^2 \, dz 
\le \frac{C}{\lambda^2} \sup_{z \in \partial_{par^*} R(t,x,2\lambda)} |u_0(z)|^2.
\end{align*}
By assumption, $u_0$ and $u$ coincide on the adjoint parabolic boundary of $R(t,x,2\lambda)$. Thus, the Carleson estimate and the strong Harnack inequality from Lemma~\ref{HL-Green-pm.lem} yield
\begin{align*}
	\sup_{z \in \partial_{par^*} R(t,x,2\lambda)} |u_0(z)|
	\leq C u(a^{-}(t,x,4 \lambda))
	\leq \inf_{z \in \Gamma} |u_0(z)| 
\end{align*}
where $\Gamma \coloneqq \{(s,y,4\lambda) \in \partial_{par^*}R(t,x,2\lambda): s \geq t + 2 \lambda^2 \}$ denotes the `forward top part' of the parabolic boundary. Altogether,
\begin{align}\label{dec-lb-jande.lem.eq1}
E_{u_0}(t,x,\lambda) \le C \frac{\inf_{z \in \Gamma} |u_0(z)| ^2}{\lambda^2}.
\end{align}

Next, we will replace the infimum in a quantitative way by the value of $u_0$ at an interior point of $R(t,x,2\lambda)$. To this end, let $\Delta \subset \Gamma$ be the parabolic boundary cylinder of radius $\lambda$ around $(t+3\lambda^2,x, 4\lambda)\in \Gamma$ and let $\varphi \in C_c^\infty(\ree_{+})$ be such that $1_{\frac{1}{2} \Delta} \leq \varphi \leq 1_\Delta$ on $\Gamma$. Let $h$ be the weak solution to
\begin{align*}
\cL_0^{*} h &= 0 \quad \text{in } R(t,x,2\lambda), \\
h &= \varphi \quad \text{on } \partial_{par^*} R(t,x,2\lambda)
\end{align*}
that exists by constant-coefficient theory (e.g.~\cite[Chapter 7, Theorem~3]{Evans}). By the maximum principle, $0 \leq h \leq 1$. Thus, the boundary H\"older estimate (Lemma~\ref{bhc.lem}) applied to $1-h$ yields $	h(t+3\lambda^2,x,(4-\eps)\lambda) \geq 1/2$ for some $\eps \in (0,2)$ that depends only on $M_0$ and dimension. Finally, the Harnack inequality implies $h(a^+(t,x,\lambda)) \geq C>0$. On the other hand, $(\inf_{z \in \Gamma} |u_0(z)|) h \leq u$ on $\partial_{par^*} R(t,x,2\lambda)$ and by the maximum principle the inequality persists at $a^+(t,x,\lambda)$ in the interior. Going back to \eqref{dec-lb-jande.lem.eq1}, we obtain
\begin{align*}
	E_{u_0}(t,x,\lambda) \le C \frac{u_0(a^+(t,x,\lambda))^2}{ \lambda^2 h(a^+(t,x,\lambda))^2} \leq C \frac{u_0(a^+(t,x,\lambda))^2}{ \lambda^2} .
\end{align*}

We conclude by an upper bound for $u_0(a^{+}(t, x,\lambda))/\lambda$: Applying the boundary Harnack inequality (Lemma~\ref{bdryharnack.lem}) to the solutions $u_0$ and $\lambda$ of $\cL_0^{*}$, we find
\[
\frac{\lamprime/2}{u_0(t,x,\lamprime/2)} \le C \frac{\lambda}{u_0(a^{+}(t,x,\lambda))}
\]
and by the mean value theorem and Lemma~\ref{constderests.lem},
\begin{align*}
	\frac{u_0(a^{+}(t,x,\lambda))^2}{\lambda^2}
	\le C \biggl( \sup_{\mu \in (0, \lamprime/2)} \partial_{\mu} u_0(t,x,\mu) \biggr)^{1/2}
	\le  C E_{u_0}(t,x,\lamprime). & \qedhere
\end{align*}
\end{proof}
%%%%%%%%%%%%%%%%%%%%%%%%%%%%%%%%%%%%%%%%%%%%%%%%%%%%%%%%%%%%%%%%
\subsection{Estimates for differences of adjoint solutions}
\label{Estimates for differences of solutions}
In order to extend Lemma~\ref{dec-lb-jande.lem} to certain solutions to (adjoint) equations with variable coefficients, we are going to control the error terms that come from the difference of two solutions. Throughout this subsection, $R= R(t,x,\lambda) \subset \ree_{+}$ is a Carleson region, $A$ is an $M_0$-elliptic, $B$ is an $\varepsilon_0$-small drift, $A_0$ is an $M_0$-elliptic constant matrix and we work with global weak solutions to the equations
\begin{align}\label{Estimates for differences of solutions.equ0}
	-\partial_t u_0 - \div A_0^{*} \nabla u_0  &= 0 \quad \text{in } R, \\
	 u_0 &= 0 \quad \text{on } Q^n(R) \notag
\end{align}
and
\begin{align}\label{Estimates for differences of solutions.equ}
		-\partial_t u - \div A^{*} \nabla u + \div( Bu ) &= 0 \quad \text{in } R,\\
		u-u_0 &= 0 \quad \text{on } \partial_{par^*} R, \notag
\end{align}
where the boundary values are understood in the Sobolev sense and both solutions vanish continuously on $Q^n(R)$.
%%%%%%%%%%%%%%%%%%%%%%%%%%%%%%%%%%%%%%%%%%%%%%%%%%%%%%%%%%%%%%%%
We start with the following simple and classical perturbation estimate, see for instance \cite{Dahlpert} and \cite{DLM} for related estimates in the elliptic setting.

\begin{lemma}
\label{dahlest.lem}
Let $M_0$ be given. There exist $\varepsilon_0, C > 0$, depending only on $M_0$ and dimension, such that in the setting above 
\begin{align}\label{dahlest.eq}
\begin{split}
&\int_{R} |\nabla(u - u_0)|^2 \, dz \\
&\le C \min \left(\int_{R} |A- A_0|^2 |\nabla u|^2 + |B|^2|u|^2 \, dz , 
\int_{R} |A- A_0|^2 |\nabla u_0|^2 + |B|^2|u_0|^2\, dz \right) 
\end{split}
\end{align}
and 
\begin{equation}\label{gradsqrcomp.eq} 
C^{-1}  \int_R |\nabla u_0|^2 \, dz \leq \int_R |\nabla u|^2 \, dz  \leq C \int_R |\nabla u_0|^2 \, dz. 
\end{equation}
\end{lemma}
\begin{proof}
%%%%%%%%%%%%%%%%%%%%%%%%%%%%%%%%%%%%%%%%%%%%%%%%%%%%%%%%%%%%%%%%
The estimate \eqref{gradsqrcomp.eq} follows from \eqref{dahlest.eq} by the triangle inequality, boundedness of $A-A_0$ and the Hardy estimate for the drift term discussed in Section~\ref{subsec.weak.sol}.

As for \eqref{dahlest.eq}, we formally use $w \coloneqq u_0 - u$ as a test function. To make this precise, we write $R = I \times \Sigma$, where $I = (a,b)$ is a time interval and $\Sigma$ a spatial cube, and set
\[
w_h(t,x,\lambda) \coloneqq \frac{1}{2h} \int_{t-h}^{t+h} 1_{(a+2h,b-2h)}(s) \, w(s,x,\lambda) \, ds.
\] 
In the limit as $h \to 0$, we have $w_h \to w$ in $L^{2}((a,b);W^{1,2}_0(\Sigma))$ by the maximal function theorem. Moreover, $w_h$ vanishes for $t = a$ and $t=b$. Since in addition $w_h \in W^{1,2}((a,b); L^2(\Sigma))$, this function is a valid test function for \eqref{Estimates for differences of solutions.equ0}, \eqref{Estimates for differences of solutions.equ}. (This approximation argument relies again on $Bu, Bu_0 \in  L^2(R)$.)

Now, by ellipticity of $A$,
\begin{align} \label{dahlest.eq1}
\begin{split}
M_0^{-1}& \int_{R} |\nabla w|^2 \, dz \\
&\leq \int_{R} A^*\nabla w \cdot \nabla w  \,dz \\
&= \lim_{h\to 0} \int_{R} A^*\nabla (u_0-u)  \cdot \nabla w_{h} \, dz\\
&=  \lim_{h\to 0} \int_{R} A_0^*\nabla u_0 \cdot \nabla w_{h} + (A^*-A_0^*) \nabla u_0 \cdot \nabla w_{h} - A^* \nabla u \cdot \nabla w_h \, dz \\
&= - \int_R Bu \nabla w \, dz + \int_R (A^* - A_0^*) \nabla u_0 \cdot \nabla w   - \lim_{h \to 0} \int_R w \partial_t w_h \, d z \\
&\eqqcolon -\I + \II  - \lim_{h \to 0} \III_h,
\end{split}
\end{align}
where the penultimate step follows by  \eqref{Estimates for differences of solutions.equ0}, \eqref{Estimates for differences of solutions.equ} and the convergence properties of $w_h$. We are going to estimate all terms in a way that produces the term with $u_0$ on the right-hand side of \eqref{dahlest.eq}. By Cauchy's inequality and Hardy's inequality,
\begin{align*}
	|\I|
	&\leq \frac{1}{2 M_0} \int_R |\nabla w|^2 \, dz + \frac{M_0}{2} \int_R |Bu|^2 \, dz \\
	&\leq \biggl(\frac{1}{2M_0} + 4 \eps_0 M_0\biggr) \int_R |\nabla w|^2 \, dz + M_0 \int_R |Bu_0|^2 \, dz 
\end{align*}
and the first term can be absorbed to the left in \eqref{dahlest.eq1} for $\eps_0 \coloneqq (16M_0^2)^{-1}$. The bound for $\II$ follows directly by Cauchy's inequality. To estimate $\III_{h}$, we write $w_{\pm h}(t,x,\lambda) \coloneqq w(t \pm h,x,\lambda)$ and note that by definition and a change of variables,
\begin{align*}
\III_{h}
& = \frac{1}{2h} \biggl( \int_{a+3h}^{b-3h} \iint_\Sigma w (w_h - w_{-h}) \, dx \, d\lambda \, dt 
 - \int_{b-3h}^{b-h} \iint_\Sigma w w_{-h} \, dx \, d\lambda \, dt \\
&\quad +  \int_{a+h}^{a+3h} \iint_\Sigma w w_h \, dx \, d\lambda \, dt \biggr)\\
&=  \frac{1}{2h}  \biggl(\int_{a+h}^{a+2h} \iint_\Sigma w w_h \, dx \, d\lambda \, dt -  \int_{b-3h}^{b-2h} \iint_\Sigma w w_h \, dx \, d\lambda \, dt \biggr).
\end{align*} 
Since $w \in C([a,b];L^{2}(Q))$ vanishes for $t=b$, we conclude that 
\begin{align*}
\lim_{h\to 0} \III_h = \frac{1}{2} \iint_\Sigma |w(a,x,\lambda)|^2 - |w(b,x,\lambda)|^2 \, dx \, d\lambda \geq 0.
\end{align*}
Going back to \eqref{dahlest.eq1}, this completes the proof of \eqref{dahlest.eq} with $u_0$ on the right. 

In order to end up with $u$, we repeat the argument but  use ellipticity of $A_0^*$ in \eqref{dahlest.eq1}.
\end{proof}
%%%%%%%%%%%%%%%%%%%%%%%%%%%%%%%%%%%%%%%%%%%%%%%%%%%%%%%%%%%%%%%%
The next two lemmas provide estimates for integrals as on the right-hand side of \eqref{dahlest.eq} but on slightly smaller boxes. They come with additional decay from the weak-DKP condition. 

The case of a solution to a constant coefficient equation is particularly simple. The main work horse here is Lemma~\ref{constderests.lem}.

\begin{lemma}\label{unoughtermsOK.lem}
Let $M_0$ and $\theta \in (0,1)$ be given and specialize to $A_0 \coloneqq \la{A}_{R}$. There exist $\varepsilon_0, C > 0$ depending only on $M_0, \theta, n$ such that
\begin{align*}
\fint_{\theta R}  |A - A_0|^2 |\nabla u_0|^2 + |B|^2 |u_0|^2 \, dz
\leq C \alpha_{A,B}(R)^2 \fint_{R}|\nabla u|^2 \, dz.
\end{align*}
\end{lemma}
%%%%%%%%%%%%%%%%%%%%%%%%%%%%%%%%%%%%%%%%%%%%%%%%%%%%%%%%%%%%%%%%
\begin{proof}
By Lemma~\ref{constderests.lem} followed by \eqref{gradsqrcomp.eq}, 
\[
\sup_{z\in \theta R} |\nabla u_0(z)|^2 \le C   \fint_{R} |\nabla u_0|^2 \, dz  \le C   \fint_{R} |\nabla u|^2 \, dz
\]
and thus
\[
\fint_{\theta R}|A - A_0|^2 |\nabla u_0|^2 \, dz
\le C \biggl(\fint_{ R} |A - A_0|^2 \, dz \biggr) \biggl( \fint_{R}|\nabla u|^2 \, dz \biggr),
\]
which is half of the claim. By the mean value theorem, Lemma~\ref{constderests.lem} and \eqref{gradsqrcomp.eq}, 
\[
\biggl(\sup_{z \in \theta R} \frac{u_0(z)}{z_{n+1}}\biggr)^2
\le \sup_{z \in \theta R} |\partial_{n+1} u_0(z)|^2
\le C \fint_{R} |\nabla u_0|^{2} \, dz
\le C \fint_{R} |\nabla u|^{2} \, dz
\] 
and we obtain the second half of the claim in the same fashion:
\begin{align*}
 \fint_{\theta R} |B|^2 |u_0|^2\, dz
 \leq C  \left( \fint_{R} |B|^2 z_{n+1}^2 \, dz  \right) \fint_{R}|\nabla u|^2 \, dz. &\qedhere
\end{align*}
\end{proof}
%%%%%%%%%%%%%%%%%%%%%%%%%%%%%%%%%%%%%%%%%%%%%%%%%%%%%%%%%%%%%%%%
We also need estimates analogous to Lemma~\ref{unoughtermsOK.lem} but with $u$ instead of $u_0$ on the left. Since the average on the left in Lemma~\ref{unoughtermsOK.lem} is bounded by that of $|\nabla u_0|^2$, controlling the difference of gradients is enough. The use of Lemma~\ref{constderests.lem} is now replaced by a variant of the Calder\'on--Zygmund estimate and the boundary reverse H\"older inequality for the gradient of solutions.  When specialized to the elliptic setting,
this argument gives a notable improvement over \cite[Lemma~3.19]{DLM}.

\begin{lemma}\label{diffOK.lem}
Let $M_0$ and $\theta \in (0,1)$ be given, specialize to $A_0 \coloneqq \la{A}_R$, and assume that $u_0$, $u$ solve \eqref{Estimates for differences of solutions.equ0}, \eqref{Estimates for differences of solutions.equ} also on $\theta^{-1} R$. Then there exist $\varepsilon_0, C > 0$ depending only on $M_0, \theta$ and dimension, such  that 
\begin{equation*}
\fint_{\theta^2 R} |\nabla (u-u_0)|^2 \, dz \leq C\alpha_{A,B}(R)^{2} \fint_{\theta^{-1} R} |\nabla u|^{2}  \, dz.
\end{equation*}
\end{lemma}
%%%%%%%%%%%%%%%%%%%%%%%%%%%%%%%%%%%%%%%%%%%%%%%%%%%%%%%%%%%%%%%%
\begin{proof}
Let $w \coloneqq u-u_0$. We shall first prove the bound
\begin{align}\label{diffOK.lem.eq1}
	\left(\fint_{\theta^{2} R} |\nabla w|^{2} \, dz \right)^{1/2} 
	\le  \frac{C}{\lambda} \fint_{\theta R}  |w| \, dz + C \alpha_{A,B}( R) \left( \fint_{R} |\nabla u|^{2}  \, dz \right)^{1/2}
\end{align}
To this end, we view $w$ as a global solution to the equation with variable coefficients
\begin{align}
\label{eq1forw.eq}
\begin{split}
  -\partial_t w  - \div A^{*} \nabla w &= \div F  \quad \text{in $R$}, \\
    w &= \makebox[\widthof{$\div F$}][l]0  \quad \text{on $\partial_{par^*}R$,} 
\end{split}
\end{align}
where
\[
F \coloneqq F_1 + F_2 + F_3, \quad F_1 \coloneqq - Bw, \quad F_2 \coloneqq - Bu_0,  \quad   F_3 \coloneqq (A^*-A_0^*) \nabla u_0 .
\] 

If $Q=Q(z,r)$ is an interior parabolic cylinder with $4 Q = Q(z,4r) \subset R$, then for $q \coloneqq 2(n+2)/n$ and $p \coloneqq2$  the reverse H\"older inequality for local solutions (e.g.\ \cite[Proposition~4.4]{ABES_JMPA}) yields
\begin{equation} \label{rhi-solution.eq}
\left(\fint_{Q} |w|^q  \, dz \right)^{1/q}
\leq C   \left( \fint_{2Q} |w|^{p} \, dz \right)^{1/p}  +  r \left(\fint_{2Q} |F_2|^2  + |F_3|^{2} \, dz \right) ^{1/2}.
\end{equation}
The term $F_1$ does not appear explicitly since $4Q \subset R$ implies $|F_1| \leq C \eps_0 r|w|$ on $2Q$. If $Q = R(t,x,r)$ is a Carleson box with $4Q = R(t,x,4 r) \subset R$, then \eqref{rhi-solution.eq} also holds for the same exponents $q,p$ --- in this case, it is precisely the statement of the boundary reverse H\"older inequality (Lemma~\ref{bdry-rhi-solution.lem}) if we write $\div F_1$ on the left-hand side of the equation. The collection $\mathcal{Q}$ of boxes, for which we know \eqref{rhi-solution.eq} with these $q,p$ so far has a covering property: every  $Q \in \mathcal{Q}$ can be covered by $Q_1,\ldots,Q_N \in \mathcal{Q}$ where $N$ only depends on dimension, $Q_i$ is comparable to $Q$ in size and $4Q_i \subset 2Q$. Thus, the classical argument of \cite[Theorem~B.1]{MR3565414} shows that in \eqref{rhi-solution.eq} the right-hand exponent can be lowered to $p=1$ and of course we can lower $q$ to $q=2$ by H\"older's inequality. Finally, we cover $\theta^{3/2} R$ by sufficiently small boxes from $\mathcal{Q}$ and sum up \eqref{rhi-solution.eq} with $q=2$, $p=1$ to conclude
\begin{equation*}
	\left(\fint_{\theta^{3/2} R} |w|^2 \, dz \right)^{1/2}
	\leq C  \fint_{\theta R} |w| \, dz +  \lambda \left(\fint_{\theta R} |F_2|^2  + |F_3|^{2} \, dz \right) ^{1/2}.
\end{equation*}
The claim \eqref{diffOK.lem.eq1} follows by the boundary Caccioppoli inequality (Lemma~\ref{caccioppoli.lem}) to the left and Lemma~\ref{unoughtermsOK.lem} to control $|F_2|^2  + |F_3|^{2}$ on the right.

With \eqref{diffOK.lem.eq1} at hand, we shall complete the argument by proving
\begin{align}\label{diffOK.lem.eq2}
		\frac{1}{\lambda} \fint_{R}  |w| \, dz \leq  C \alpha_{A,B}(R) \left( \fint_{\theta^{-1} R} |\nabla u|^{2}  \, dz \right)^{1/2}.
\end{align}
To this end, we rearrange \eqref{eq1forw.eq} and view $w$ as a global solution to the equation
\begin{align}\label{u-u0 constant coefficient.eq}
\begin{split}
  -\partial_t w  - \div A_0^* \nabla w &= -\div Bu + \div  (A^*-A_0^*) \nabla u \quad \text{in $R$}, \\
    w &= \makebox[\widthof{$-\div Bu- \div  (A^*-A_0^*) \nabla u$}][l]0 \quad \text{on $\partial_{par}^{*}R$} 
\end{split}
\end{align}
with constant coefficients. By H\"older's inequality, Poincar\'e's inequality and the Calder\'on--Zygmund estimates (see \cite[Chapter IV]{MR0241822} or \cite[Theorem~4.1]{MR3985550}), 
\begin{align*}
	\frac{1}{\lambda} \fint_{R}  |w| \, dz
&\leq \frac{1}{\lambda} \biggl(\fint_{R}  |w|^q \, dz \biggr)^{1/q} \\
&\leq C \biggl(\fint_{R} |\nabla w|^{q} \, dz \biggr)^{1/q} \\
&\leq C \biggl(\fint_{R} |B|^{q}|u|^{q} \, dz \biggr)^{1/q} + C \biggl(\fint_{R}  |A-A_0|^{q}|\nabla u|^{q} \, dz \biggr)^{1/q}
\eqqcolon \I + \II
\end{align*}
holds for any $q > 1$ and a constant $C$ depending also on $q$. We fix $p>2$ to be the exponent provided by Lemma~\ref{RHforgflem.lem}, take $q$ such that $1/q = 1/2+1/p$ and estimate $\I$ and $\II$ by the right-hand side of \eqref{diffOK.lem.eq2} as follows.
First, by H\"older's and Hardy's inequalities,
\begin{align*}
\I \leq C \left( \fint_{R} |B|^{2}z_{n+1}^{2} \, dz \right)^{1/2}\left( \fint_{R} \frac{|u|^{p}}{z_{n+1}^{p}}\, dz \right)^{1/p} 
\leq C \alpha_B(R)^{1/2} \left( \fint_{R} |\nabla u|^{p} \, dz \right)^{1/p}
\end{align*}
and the correct bound follows from by Lemma~\ref{RHforgflem.lem}.
Similarly,
\begin{align*}
\II 
\leq C \left( \fint_{R} |A-A_0|^{2} \, dz  \right)^{1/2}\left( \fint_{R} |\nabla u|^{p} \, dz \right)^{1/p}  
\leq C \alpha_{A}(R)^{1/2} \left( \fint_{\theta^{-1}R} |\nabla u|^{2}  \, dz \right)^{1/2}. &\qedhere
\end{align*}
\end{proof}
%%%%%%%%%%%%%%%%%%%%%%%%%%%%%%%%%%%%%%%%%%%%%%%%%%%%%%%%%%%%%%%%
\subsection{Estimates for Green functions}
\label{Estimates for Green function}
The following estimates for Green functions are a perturbed version of Lemma~\ref{dec-lb-jande.lem}.

\begin{lemma}\label{non-constant-DLM-main.lem}
Let $M_0$ be given. There exists $\varepsilon_0 > 0$, $\kappa_0 \ge 40$ and $C \ge 1$ such that whenever
$G$ is the Green function for an $(M_0,\varepsilon_0)$-parabolic operator
and $u \coloneqq G(a^{+}(t_0,x_0,\kappa \lambda_0),\bigdot)$ for some $(t_0,x_0,\lambda_0) \in \ree_{+}$ and $\kappa \geq \kappa_0$,  then the following estimates hold for all $(t,x,\lambda) \in R(t_0,x_0,\lambda_0)$ and all $\lamprime \in (0,\lambda]$:
\begin{enumerate}
  \item \emph{Decay of $J$}:
  \begin{equation}
    \label{non-constant-DLM-main-1.eq}
J_u(t,x,\lamprime) \le C \biggl( \frac{\lamprime}{\lambda} \biggr)^{2} J_u(t,x,\lambda) + C\biggl( \frac{\lambda}{\lamprime} \biggr)^{n+2} \alpha_{A,B}(t,x,\lambda)^2 E_u(t,x,\lambda).
  \end{equation}
  \item \emph{Non-degeneracy of $E$}:
  \begin{equation}
    \label{non-constant-DLM-main-2.eq}
  E_u(t,x,\lambda) \le C E_u(t,x,\lamprime) + C \biggl( \frac{\lambda}{\lamprime} \biggr)^{n+2} \alpha_{A,B}(t,x,\lambda)^2 E_u(t,x,\lambda).
\end{equation}
\end{enumerate}
\end{lemma}
%%%%%%%%%%%%%%%%%%%%%%%%%%%%%%%%%%%%%%%%%%%%%%%%%%%%%%%%%%%%%%%%
\begin{proof}
We begin with the first item. As in the proof of Lemma~\ref{dec-lb-jande.lem} we can assume $\lamprime \leq \lambda/4$.
Let $A_0 \coloneqq \la{A}_{R(t,x,\lambda/2)}$ and let $u_0$ solve
\begin{align*}
  -\partial_t u_0 - \div A_0^* \nabla u_0 &= 0 \quad \text{in $R(t,x,\lambda/2)$},\\
    u_0 &= u \quad \text{in $\partial_{par^*} R(t,x,\lambda/2)$} .
\end{align*}
Such a solution exists --- it is given by $u_0 \coloneqq u-w$, where $w$ is the solution to the problem \eqref{u-u0 constant coefficient.eq} with constant coefficients on $R\coloneqq R(t,x,\lambda/2)$ and the latter is solvable (e.g.~\cite[Chapter 7, Theorem~3]{Evans}). Now, we split
\begin{align} \label{non-constant-DLM-main.lem.eq1}
J_u(t,x,\lamprime) \le 2J_{u_0}(t,x,\lamprime)+ 2J_{u-u_0}(t,x,\lamprime)
\end{align}
and estimate both energies separately. By Lemma~\ref{dec-lb-jande.lem},  the triangle inequality and Lemma~\ref{diffOK.lem},
\begin{align*}
J_{u_0}(t,x,\lamprime) 
&\leq C \biggl( \frac{\lamprime}{\lambda} \biggr)^{2} J_{u_0}(t,x,\lambda/4)  \\
&\leq  C \biggl( \frac{\lamprime}{\lambda} \biggr)^{2} \bigl(J_{u}(t,x,\lambda/4) + E_{u-u_0}(t,x,\lambda/4) \bigr)\\
&\leq C \biggl( \frac{\lamprime}{\lambda} \biggr)^{2} \bigl(J_{u}(t,x,\lambda/4) +  \alpha_{A,B}(t,x,\lambda)^{2} E_u(t,x,\lambda)\bigr).
\end{align*}
To bound the remaining term in \eqref{non-constant-DLM-main.lem.eq1},
we estimate 
\begin{align*}
J_{u-u_0}(t,x,\lamprime) 
&\leq E_{u-u_0}(t,x,\lamprime) \\
&\leq \left( \frac{\lambda}{\lamprime} \right)^{n+2}E_{u-u_0}(t,x,\lambda/4) \\
&\leq \biggl( \frac{\lambda}{\lamprime} \biggr)^{n+2}\alpha_{A,B}(t,x,\lambda)^{2} E_u(t,x,\lambda), 
\end{align*}
where the last estimate follows from Lemma~\ref{diffOK.lem} as before.
This yields the first item.

As for the second item, Lemma~\ref{HL-Green-pm.lem} and the triangle inequality yield
\begin{align}\label{non-constant-DLM-main.lem.eq2}
E_u(t,x, \lambda) 
	\leq C E_u(t,x,\lambda/4) 
	\leq C \bigl( E_{u_0}(t,x,\lambda/4)  +  E_{u-u_0}(t,x,\lambda/4)  \bigr).
\end{align}
In the proof of the first assertion, we have already seen that
\[
E_{u-u_0}(t,x, \lambda/4)
 \le C \alpha_{A,B}(t,x,\lambda)^2 E_u(t,x,\lambda),
\]
which is of the desired form. To estimate the first term on the right of \eqref{non-constant-DLM-main.lem.eq2}, we use Lemma~\ref{dec-lb-jande.lem}, the triangle inequality and once again Lemma~\ref{diffOK.lem}:
\begin{align*}
E_{u_0}(t,x,\lambda/4)
&\leq C E_{u_0}(t,x,\lamprime) \\
&\leq C \bigl(E_{u}(t,x,\lamprime) + E_{u-u_0}(t,x,\lamprime) \bigr) \\
&\leq C E_{u}(t,x,\lamprime) + C \biggl( \frac{\lambda}{\lamprime} \biggr)^{n+2} E_{u-u_0}(t,x,\lambda/4) \\
&\leq C E_{u}(t,x,\lamprime) + C \biggl( \frac{\lambda}{\lamprime} \biggr)^{n+2} \alpha_{A,B}(t,x,\lambda)^2 E_{u}(t,x,\lambda). \qedhere
\end{align*}
\end{proof}
%%%%%%%%%%%%%%%%%%%%%%%%%%%%%%%%%%%%%%%%%%%%%%%%%%%%%%%%%%%%%%%%
We are in  position to state and prove the first central result that was largely inspired by the elliptic counterpart in \cite{DLM}. Under an a priori smallness on the coefficients in terms of the $\alpha_{A,B}$-numbers, we obtain a quantitative Carleson measure estimate for the energy ratios
\begin{align*}
	\beta_u(t,x,\lambda) = \left(\frac{J_{u} (t,x,\lambda)}{E_{u} (t,x,\lambda)}\right)^{1/2}
\end{align*}
from Definition~\ref{JandEdefs.def}. In this context, we also use the measure
\begin{align*}
	\mu_{\beta_u} \coloneqq \beta_u(t,x,\lambda)^2 \, \frac{dt \, dx \, d \lambda}{\lambda}.
\end{align*}

\begin{theorem}\label{green-main-estimate.thm}
Given $M_0$, let $\varepsilon_0>0$ and $\kappa_0 \geq 40$ be as in Lemma~\ref{non-constant-DLM-main.lem}. There exists constants $\delta_0 >0$, $\theta_0 \in (0, 1/4]$, $\gamma > 0$ and $C$ such that whenever $G$ is the Green function for an $(M_0,\varepsilon_0)$-parabolic operator
and $u \coloneqq G(a^{+}(t_0,x_0,\kappa \lambda_0),\bigdot)$ for some $(t_0,x_0,\lambda_0) \in \ree_{+}$ and $\kappa \geq \kappa_0$,  then the uniform smallness
\[
\sup_{(s,y,\mu) \in R(t,x,\lambda)  } \alpha_{A,B}(s,y,\mu)^{2} \le \delta_0
\]
for fixed $(t,x,\lambda) \in R(t_0,x_0,\lambda_0)$  implies for all $\theta \leq \theta_0$ the
\begin{enumerate}
	\item pointwise Carleson bound
	\[
	\frac{1}{(\theta \lambda)^{n+1}} \iiint_{R(t,x,\theta \lambda)}\beta(s,y,\mu)^{2} \, \frac{ds \, dy \, d\mu}{\mu}  \le C \bigl( \theta^{\gamma} + \no{\nu_{A,B}}_{\mathcal{C} (R(t,x, 9\lambda))}\bigr) 
	\]
	\item and the local Carleson measure estimate
	\[
	\no{\mu_{\beta_u}}_{\mathcal{C}(R(t,x,\theta^2 \lambda))} \leq  C \bigl( \theta^{\gamma} + \no{\nu_{A,B}}_{\mathcal{C} (R(t,x, \lambda))}\bigr)
	\]
\end{enumerate}
\end{theorem}
%%%%%%%%%%%%%%%%%%%%%%%%%%%%%%%%%%%%%%%%%%%%%%%%%%%%%%%%%%%%%%%%
\begin{proof}
Assertion (2) follows from (1) applied to subregions: Indeed, we only have to take $\theta_0 < 1/36$ in order to guarantee that $R(t',x',\theta \lambda') \subset R(t,x,\theta^2 \lambda)$  implies $R(t',x',9 \lambda') \subset R(t,x,\lambda)$. Hence, it suffices to prove (1). 

To this end, let $\sigma \in (0,1/4]$ be a scaling parameter that we shall fix momentarily. The other constants in question will be taken as $\delta_0 \coloneqq \sigma^{n+2}/2$, $\theta_0 \coloneqq \sigma$ and $\gamma \coloneqq \log(1/2)/\log(\sigma)$. Throughout the proof let $(s,y,\mu) \in R(t, x, \lambda) \eqqcolon R$. With the estimates above we closely follow \cite[Section 4]{DLM}.

From \eqref{non-constant-DLM-main-1.eq}, \eqref{non-constant-DLM-main-2.eq} and the assumption, we obtain
\begin{align*}
J_u(s,y,\sigma \mu) &\leq C_1 \sigma^2  J_u(s,y,\mu) + C_2 \sigma^{-n-2} \alpha_{A,B}(s,y,\mu)^2 E_u(s,y,\mu), \\
E_u(s,y,\sigma \mu) &\geq C_3(1-  \delta_0 \sigma^{-n-2}) E_u(s,y,\mu).
\end{align*}
We decide on $\sigma^2 \leq C_3/(4C_1)$, so that when dividing the two inequalities, we find that $\beta$ shrinks according to the rule
\begin{align}\label{green-main-estimate.thm.beta-shrinks}
\beta(s,y, \sigma \mu)^2 \leq \frac{1}{2} \beta(s,y,\mu)^2 + C \alpha_{A,B}(s,y,\mu)^2.
\end{align}
For brevity, let us now write
\begin{align*}
b(\mu) &\coloneqq  \int_{Q^{n}(\theta R)} \beta(z+(0,0,\mu))^{2} \, dz,\\
a(\mu) &\coloneqq \int_{Q^{n}(\theta R)} \alpha_{A,B}(z+(0,0,\mu))^{2} \, dz.
\end{align*}
We need to prove that
\begin{align}\label{green-main-estimate.thm.goal}
\I \coloneqq \int_{0}^{2\theta \lambda} b(\mu) \,  \frac{d\mu}{\mu} 
\leq C(\theta \lambda)^{n+1}  (\theta^\gamma + \no{\nu_{A,B}}_{\mathcal{C} (R)}).
\end{align}
To this end, we split the integral $\I$ at height $\mu = 2\sigma \theta \lambda$ and estimate both terms separately. By a change of variables and \eqref{green-main-estimate.thm.beta-shrinks},
\begin{align*}
\int_{0}^{2\sigma \theta \lambda} b(\mu) \,  \frac{d\mu}{\mu}
= \int_{0}^{2\theta \lambda} b(\sigma \mu) \,  \frac{d\mu}{\mu}
&\leq \frac{1}{2} \I + 
 C \int_{0}^{2\theta \lambda} a (\mu) \, \frac{d\mu}{\mu}\\
& \leq \frac{1}{2} \I + C (\theta \lambda)^{n+1} \no{\nu_{A,B}}_{\mathcal{C} (R)}. 
\end{align*}
For the other piece, we let $k \geq 1$ be the integer with $\sigma^{k+1} \le \theta < \sigma^{k}$, so that by definition of $\gamma$ we have $2^{-(k+1)} \leq \theta^\gamma$. A change of variables and $k$-fold iteration of \eqref{green-main-estimate.thm.beta-shrinks} yield,
\begin{align*}
\int_{2\sigma \theta \lambda}^{2\theta \lambda} b(\mu) \,  \frac{d\mu}{\mu}
\leq \int_{2\sigma^2 \lambda}^{ 2\lambda} b( \sigma^{k} \mu) \, \frac{d\mu}{\mu}
&\leq  \int_{2\sigma^2  \lambda}^{ 2\lambda} \biggl(2^{-k}  b(  \mu) +  C \sum_{j=0}^{k}2^{-j} a(\sigma^{k-j} \mu)\biggr)  \,  \frac{d\mu}{\mu} \\
&\leq C(\theta \lambda)^{n+1}  \bigl(\theta^\gamma + \no{\nu_{A,B}}_{\mathcal{C} (9R)}\bigr),
\end{align*}
where the last step follows from the uniform pointwise bounds $\beta(s,y,\mu) \leq 2$ (triangle inequality) and $\alpha_{A,B}(s,y,\mu) \leq C \no{\nu_{A,B}}_{\mathcal{C} (9R)}$ (Proposition~\ref{whitney-carleson.prop}). The previous two estimates lead to \eqref{green-main-estimate.thm.goal} under the a priori assumption that $\I$ is finite. 

The a priori assumption can be removed by repeating the above with $1_{(\varepsilon,\infty)}(\mu)b(\mu)$ in place of  $b(\mu)$ for some fixed $\varepsilon > 0$ and passing to the limit as $\varepsilon \to 0$.
\end{proof}
%%%%%%%%%%%%%%%%%%%%%%%%%%%%%%%%%%%%%%%%%%%%%%%%%%%%%%%%%%%%%%%%
\subsection{Estimates for the parabolic measure}
\label{Estimates for the parabolic measure}
We conclude this section by giving a lower bound for the energy ratio $\beta$
in terms of the parabolic measure.
This is a straightforward generalization of an analogous result for elliptic measures in \cite{BTZ}. For functions $f$ defined in the variables $(t,x) \in \R^n$ we use the parabolic dilations
\begin{align}\label{parabolic-dilation.eq}
	f_\lambda(t,x) \coloneqq \frac{1}{\lambda^{n+1}} f\left( \frac{t}{\lambda^{2}}, \frac{x}{\lambda} \right) \qquad (\lambda > 0).
\end{align}

\begin{lemma}\label{parab-measure-main.lem}
Given $M_0$, let $\varepsilon_0>0$ and $\kappa_0 \geq 40$ be as in Lemma~\ref{non-constant-DLM-main.lem} and let $\phi, \psi: \R^n \to \R$ be smooth functions with
\begin{align*}
 \phi \geq 1 \text{ on } Q^n(0,0,\tfrac{1}{2}) \quad \text{and} \quad  \phi, \partial_i \psi \in C_c^\infty(Q^{n}(0,0,\tfrac{3}{4})) \text{ for } i = 1, \ldots, n.
\end{align*}
There exists a constant $C$ such that the following holds. Let $G$ be the Green function for an $(M_0,\varepsilon_0)$-parabolic operator and set $u \coloneqq G(p,\bigdot)$, where $p \coloneqq a^{+}(t_0,x_0,\kappa\lambda_0)$ for some $(t_0,x_0,\lambda_0) \in \ree_{+}$, $\kappa \geq \kappa_0$. Then the associated parabolic measure $\omega \coloneqq \omega^p$ satisfies for all $(t,x,\lambda) \in R(t_0,x_0,\lambda_0)$ the estimates
\begin{align}
\label{energy-general.eq}
|(\partial_t \psi)_\lambda *\omega)(t,x)|  + |(\nabla_x \psi)_\lambda *\omega)(t,x)|    &\le C \sqrt{ J_u(t,x,\lambda)  +  E_u(t,x,\lambda)  \alpha_{A,B}(t,x,\lambda)^2 } , \\
\label{ratio-general.eq}
\frac{|(\partial_t \psi)_\lambda *\omega)(t,x)|}{(\phi_\lambda *\omega)(t,x)}  + \frac{  |(\nabla_x \psi)_\lambda *\omega)(t,x)|}{(\phi_\lambda *\omega)(t,x)}    &\le C  \bigl(\beta_u(t,x,\lambda)  +    \alpha_{A,B}(t,x,\lambda)  \bigr)
\end{align}   
Moreover, $(t,x,\lambda)$ on the right-hand sides can be replaced by any $(t',x',\lambda')$ that satisfies $d_{n+1}((t',x',\lambda'), (t,x,\lambda)) < \lambda/4$.
\end{lemma}
%%%%%%%%%%%%%%%%%%%%%%%%%%%%%%%%%%%%%%%%%%%%%%%%%%%%%%%%%%%%%%%%
\begin{proof}
We directly prove the general claim with $(t',x',\lambda')$ on the right. Once \eqref{energy-general.eq} has been proved, \eqref{ratio-general.eq} will follow immediately since
\begin{align*}
	(\phi_\lambda * \omega)(t,x) 
	\geq \frac{\omega(Q^n(t,x,\lambda/2))}{\lambda^{n+1}} 
	&\geq  C \frac{\omega(Q^n(t',x',\lambda'/5))}{(\lambda')^{n+1}} \\
	&\geq  C \frac{\omega(Q^n(t',x',\lambda'))}{(\lambda')^{n+1}} 
	\geq C \sqrt{E_u(t',x',\lambda')},
\end{align*}
where we have used the lower bound for $\phi$, local doubling, the CFMS-estimate and the gradient estimate from Lemma~\ref{HL-Green-pm.lem}. Hence, it suffices to prove \eqref{energy-general.eq}.
To this end, fix a smooth function $h$ of one variable with $h(0) =1$ and support in $(-1/4,1/4)$ and define the extension 
\[
\Psi(s,y,\mu) = \frac{1}{\lambda^{n+1}}\psi\left( \frac{t-s}{\lambda^{2}}, \frac{x-y}{\lambda} \right) h\left( \frac{\mu}{\lambda} \right).
\]
We will frequently use that the non-zero values of the derivatives of this function in the upper-half-space are contained in $R(t',x',\lambda')$ and that these derivatives are uniformly bounded by certain powers of $\lambda \approx \lambda'$. The Riesz formula \eqref{Rieszformeq.eq} applied to $\partial_i \Psi$ for $i = 1, \ldots, n$ yields
\begin{align*}
\partial_{i} (\omega * \psi_{\lambda})(t,x)
&= -\int_{\ree_{+}} A^{*}\nabla u \cdot \nabla \partial_i \Psi \, dz
-  \int_{\ree_{+}} Bu \cdot \nabla \partial_i \Psi \, dz
- \int_{\ree_{+}}  u \partial_t \partial_i \Psi \, dz \\
&\eqqcolon -\I(i) - \II(i) - \III(i).
\end{align*}
We estimate these three terms separately. Note carefully that here the left-hand side differs from the left-hand side in \eqref{energy-general.eq} by a factor of $\lambda^2$ when $i=1$ and a factor $\lambda$ when $i=2,\ldots,n$. 

As the derivative of a compactly supported function has integral zero, we can write
\[
\III(i) = \int_{\ree_{+}} \bigl( u - z_{n+1} \la{\partial_{n+1}u}_{R(t',x',\lambda')}\bigr) \partial_t \partial_i \Psi \, d z.
\]
and a desirable bound in terms of $J_u(t',x',\lambda')^{1/2}$ follows from H\"older's and Poincar\'e's inequality.

By H\"older's inequality and the support properties of $\Psi$, we also obtain 
\begin{align*}
\ab{\II(i)} 
&\leq  C \lambda^{n+2} 
\|\nabla \partial_i  \Psi\|_\infty  \left(\fint_{R(t',x',\lambda')} \frac{|u|^{2}}{z_{n+1}^{2}} \, dz \right)^{1/2} 
\times \left(\fint_{R(t',x',\lambda')} |B|^{2} z_{n+1}^2 \, dz  \right)^{1/2}
\end{align*}
and a desirable bound in terms of $E_u(t',x',\lambda')^{1/2} \alpha_B(t',x',\lambda')$ follows from the usual Hardy's inequality.

In order to control the most substantial contribution $\I(i)$, we introduce $\ell(z) \coloneqq z_{n+1} \la{\partial_{n+1}u}_{R(t',x',\lambda')}$ and find
\begin{align*}
\I(i) 
&= \int_{\ree_{+}} A^{*}\nabla (u - \ell) \cdot \nabla \partial_i \Psi \, dz
  + \int_{\ree_{+}} (A^{*} - \la{A^{*}}_{R(t',x',\lambda')}) \nabla \ell \cdot \nabla \partial_i \Psi \, dz\\
&\quad+ \int_{\ree_{+}} \la{A^{*}}_{R(t',x',\lambda')} \nabla \ell \cdot \nabla \partial_i \Psi \, dz.
\end{align*}
By H\"older's inequality and the size and support properties of $\Psi$, the first term integral admits a desirable bound in terms of $J_u(t',x',\lambda')^{1/2}$. Likewise, a desirable bound for the second integral in terms of $\alpha_{A}(t',x',\lambda') E_u(t',x',\lambda')^{1/2}$ follows.
The third term is zero, because the integrand
\[
\la{A^{*}}_{R(t',x',\lambda')} \nabla \ell \cdot \nabla \partial_i \Psi =  \partial_i \bigl(\la{A^{*}}_{R(t',x',\lambda')} \nabla \ell \cdot \nabla  \Psi \bigr),
\]
is the (tangential) derivative of a smooth and compactly supported function.
\end{proof}
%%%%%%%%%%%%%%%%%%%%%%%%%%%%%%%%%%%%%%%%%%%%%%%%%%%%%%%%%%%%%%%%
\section{A criterion for anisotropic $A_\infty$-weights}
\label{anisotropic weights}

\noindent In this section, we provide a criterion for anisotropic (local) $A^\infty$-weights, inspired by \cite{FKP,  BES-DKP} and tailored towards the estimates for parabolic measure in Lemma~\ref{parab-measure-main.lem}. Unlike the references cited, we do not assume the weight function to be globally doubling but proceed under the following background assumption.
%%%%%%%%%%%%%%%%%%%%%%%%%%%%%%%%%%%%%%%%%%%%%%%%%%%%%%%%%%%%%%%%
\begin{definition}[Local doubling]
A locally finite Borel measure $\omega$ on a metric space $(X,d)$ is locally doubling in an open set $U$
if there exists a constant $C_\omega$ such that for all $x \in U$ and $r > 0$ such that $B(x,2r) \subset U$ it follows $\omega(B(x,2r)) \le C_\omega \omega(B(x,r))$. 
\end{definition}
%%%%%%%%%%%%%%%%%%%%%%%%%%%%%%%%%%%%%%%%%%%%%%%%%%%%%%%%%%%%%%%%
In the following, we fix an $n$-tuple $(p_1,\ldots,p_n) \in [1,\infty)^n$ and work in $\R^n$ with the metric
\begin{align*}
	d(z,w) \coloneqq \max_{1 \le j \le n} |z_j-w_j|^{1/p_j}.
\end{align*}
For cubes in $\R^n$ and Carleson boxes in $\R^{1+n}$ we use the same notation as in Section~\ref{General notation.sec}. For functions $f$ on $\R^n$ and $\lambda>0$ we introduce the anisotropic dilations
\begin{equation*}
	f_\lambda(z) \coloneqq \frac{1}{\lambda^p} \left(\frac{z_1}{\lambda^{p_1}},\ldots, \frac{z_n}{\lambda^{p_n}}\right),
\end{equation*}
where $p \coloneqq \sum_{j=1}^{n}p_j$. Earlier on in \eqref{parabolic-dilation.eq} we used $(p_1,p_2,\ldots,p_n) = (2,1,\ldots,1)$. 

The global doubling hypothesis in \cite{FKP, BES-DKP} is related to the use of heat extensions and infinite speed of propagation. Replacing heat extensions by more locally behaved averages  is the motivation of the following definition. Therein, the heat kernel $\phi$ would correspond to the profile $g(s) = e^{-s^2/2}$.

\begin{definition}[Split heat kernel]
An $(n+1)$-tuple $(\phi,\psi_1,\ldots,\psi_n)$ of smooth functions on $\R^n$ is called a split heat kernel if there is a smooth profile $g : \R \to \R$ with $1_{[-1/2,1/2]} \le g \le 1_{[-1,1]}$ and integral $1$, such that
\begin{enumerate}
	\item $\displaystyle \phi(z)= \prod_{i=1}^{n}g(z_i)$,
	\item $\displaystyle \psi_j(z) = \biggl(\int_{-\infty}^{z_{j}} g(s) \, s \, ds\biggr) \prod_{i \ne j} g(z_{i}) .$
\end{enumerate}
\end{definition}
%%%%%%%%%%%%%%%%%%%%%%%%%%%%%%%%%%%%%%%%%%%%%%%%%%%%%%%%%%%%%%%%
The split heat kernel satisfies $\partial_j \psi_j(z) = z_j \phi(z)$ for $j=1,\ldots, n$ and therefore its (anisotropic) dilations satisfy the following identity, which is reminiscent to the heat equation:
\begin{align}
\label{partialphircomput.eq}
\begin{split}
\lambda \partial_\lambda (\phi_\lambda)
 &= - \sum_{j=1}^{n} p_j \bigl(\phi_\lambda + (z_j \partial_j \phi)_\lambda\bigr) \\
 &= - \sum_{j=1}^{n} p_j  (\partial_j^2 \psi_j)_\lambda
 = - \sum_{j=1}^{n} p_j \lambda^{2p_j} \partial_j^2 \bigl((\psi_j)_\lambda \bigr).
\end{split}
\end{align}
%%%%%%%%%%%%%%%%%%%%%%%%%%%%%%%%%%%%%%%%%%%%%%%%%%%%%%%%%%%%%%%%

The following is our main result on anisotropic weights.

\begin{theorem}
\label{weights-characterization.thm}
Let $Q_0 \subset \R^n$ be a metric ball and $\omega$ be a locally doubling Radon measure on $4Q_0$ that is not identically zero on $2Q_0$. Suppose that there is a split heat kernel $(\phi, \psi_1, \ldots, \psi_n)$ and a constant $\varepsilon > 0$ such that for $j=1,\ldots,n$ and $k=1,2$,
\begin{align}
\label{locparalph.eq}
\biggl\|\frac{|(\partial_j \psi_j)_{\lambda} * \omega| \, |(\partial_j \phi)_{\lambda} * \omega|}{|\phi_{\lambda} * \omega|^{2}} \, \frac{dz \, d\lambda}{\lambda}\biggr\|_{\mathcal{C}(2Q_0)} &\le \varepsilon, \\
\label{locparbetat.eq}
\biggl\|\frac{(\partial_j^{k} \phi)_{\lambda} * \omega}{\phi_{\lambda} * \omega}\biggr\|_{L^{\infty}(R(2Q_0))} +  \biggl\|\frac{(\partial_j^{k}\psi_j)_{\lambda} * \omega}{\phi_{\lambda} * \omega}\biggr\|_{L^{\infty}(R(2Q_0))} &\le \sqrt{\varepsilon}.
\end{align}
Then $\omega$ is absolutely continuous with respect to Lebesgue measure on $2Q_0$ and there is a constant $C$ depending on $n, p_1, \ldots, p_n$ and the local doubling of $\omega$ such that
\begin{equation}
\label{Ainfsingscalefkp.eq}
0 \leq \log \left(\fint_{Q} \omega(z) \, dz \right)  - \fint_{Q} \log \omega(z) \, dz  \le C \bigl(\sqrt{\varepsilon} + \varepsilon\bigr)
\end{equation}
for all metric balls $Q \subset Q_0$.
\end{theorem}
%%%%%%%%%%%%%%%%%%%%%%%%%%%%%%%%%%%%%%%%%%%%%%%%%%%%%%%%%%%%%%%%
\begin{remark}
The additional assumption \eqref{locparbetat.eq} compensates the lack of algebraic properties of the heat kernel compared to \cite{FKP, BES-DKP}. Indeed, in these references $- \psi_j = \phi$ is the heat kernel and therefore \eqref{locparbetat.eq} follows from \eqref{locparalph.eq} via regularity estimates for solutions to the heat equation~\cite[Lemma~3.3]{BES-DKP}. 
\end{remark}
%%%%%%%%%%%%%%%%%%%%%%%%%%%%%%%%%%%%%%%%%%%%%%%%%%%%%%%%%%%%%%%%
\begin{proof} 
We extend $\omega$ to $\ree_{+}$ via the split heat kernel as
\begin{equation*}
	%\label{ainftycharproof-udef.eq}
	u(z,\lambda) \coloneqq (\phi_{\lambda} * \omega)(z), \quad v_j(z,\lambda) \coloneqq ((\psi_j)_{\lambda} * \omega)(z) \qquad (\lambda > 0).
\end{equation*}
In a first step, we control the integral in \eqref{Ainfsingscalefkp.eq} for $u(\bigdot, \delta)$ replacing $\omega$. All further assertions will follow upon carefully passing to the limit as $\delta \to 0$.

\emph{Step 1: The quantitative estimate}.  We write $Q_0 = Q(z_0,r_0)$. Let $Q = Q(w, r) \subset 2Q_0$ and $\delta \in (0, r_0)$. Here, we are going to prove
\begin{align}\label{weights-characterization.thm.eq1}
	0 \leq \log \left(\fint_{Q} u(z,\delta) \, dz \right)  - \fint_{Q} \log u(z,\delta) \, dz
	\le \begin{cases}
		C & \text{ if } \delta > r \\C (\sqrt{\varepsilon} + \varepsilon) & \text{ if } \delta \leq r
	\end{cases}.
\end{align}

The lower bound is due to Jensen's inequality. The upper bound for $\delta > r$ follows from doubling of $\omega$ since
\begin{align*}
	C_\omega^{-1} \frac{\omega(Q(w,\delta))}{\delta^p} \leq u(z,\delta) \leq  C_\omega \frac{\omega(Q(w,\delta))}{\delta^p} \qquad (z \in Q),
\end{align*} 
so that \eqref{weights-characterization.thm.eq1} follows with $C \coloneqq 3 \log C_\omega$. 

We turn to the case $\delta \leq r$. Since assumption and claim are translation and scaling invariant, we can assume without loss of generality that $Q = Q(0,1)$. Local doubling implies $\inf_{Q \times (\delta,r)} u > 0$. This will guarantee that all integrals below converge absolutely. We write the left-hand side in \eqref{weights-characterization.thm.eq1} as
\begin{align*}
\left[ \log\left( \fint_{Q} u(z,\delta) \, dz \right) - \fint_{Q} \log u(z,1) \, dz \right] + &\left[ \fint_{Q} \log u(z, 1) \, dz - \fint_{Q} \log u(z,\delta) \, dz \right] \\ 
&\eqqcolon \I + \II
\end{align*}
and start the estimation from the term $\II$. 

By the fundamental theorem of calculus and  \eqref{partialphircomput.eq}, 
\begin{align*}
\II 
&=   \fint_Q \int_\delta^1 \partial_\lambda (\log u) \, d\lambda\, dz \\
&=  \fint_Q \int_\delta^1 \frac{\partial_\lambda u}{u}  \,  d\lambda \, dz \\
&=  \fint_Q \int_\delta^1  - \sum_{j=1}^{n} \frac{p_j \partial_j^2 (v_j)}{u} \, \frac{d\lambda} {\lambda^{1-2p_j}} \, dz \\
&= - \sum_{j=1}^{n} p_j \int_\delta^1 \fint_Q  \partial_j \biggl(\frac{\partial_j v_j}{ u} \biggr)  +  \frac{\partial_j v_j \cdot \partial_j u}{|u|^2} \, dz \, \frac{d\lambda}{\lambda^{1-2p_j}} \\
&\eqqcolon - \sum_{j=1}^n p_j\II_{1,j}+p_j \II_{2,j}.
\end{align*}
Assumption \eqref{locparalph.eq} implies $|\II_{2,j}| \leq \varepsilon$ for all $j$.
For the integrals $\II_{j,1}$ we apply the fundamental theorem of calculus and then use \eqref{locparbetat.eq} to conclude
\[
\II_{j,1} \leq \int_{0}^{1} \int_{\partial Q \cap \{z_j = \pm 1\}} \biggl|\frac{\partial_j v_j}{ u} \biggr| \, dz \, \frac{d\lambda }{\lambda^{1-2p_j}}
\le C \sqrt{\varepsilon}.
\]
Altogether, we have shown that $\II \leq C(\varepsilon + \sqrt{\varepsilon})$.

Next, we decompose
\begin{align*}
\I &= 
\biggl[ \log\biggl( \fint_{Q} u(z, \delta) \, dz \biggr)  - \log \sup_{z \in Q} u(z,1) \biggr]
+ \biggl[\sup_{z \in Q} \log u(z,1) - \fint_{Q} \log u(z, 1) \, dz \biggr] \\
&\eqqcolon \I_1 + \I_2.
\end{align*}
By the mean value theorem and \eqref{locparbetat.eq} we have $\I_2 \leq C \sqrt{\varepsilon}$. Local doubling of $\omega$ also implies $\I_1 \leq C$, so that it remains to control $\I_1$ for small $\varepsilon$, say $\sqrt{\eps} \leq \min_j \frac{p_j}{2p}$. 

Writing $\I_1$ as the quotient of logarithms, we find
\begin{equation}
\label{log-in-painfty.eq}
\begin{split}
\I_1 
\leq \log_{+} \biggl(1 +  \inf_{z \in \overline{Q}} \frac{1}{u(z,1)} \fint_{Q} u(z,\delta)- u(z,1) \, dz \biggr),
\end{split}
\end{equation}
where $\log_+$ is the positive part of the logarithm function. By the fundamental theorem of calculus and \eqref{partialphircomput.eq}, the average on the right is bounded by
\begin{align*}
\fint_{Q} u(z,\delta) - u(z,1) \, dz
&= - \fint_Q \int_\delta^1 \partial_\lambda u \, d \lambda \, dz \\
&= \int_\delta^1 \fint_Q \sum_{j=1}^{n} p_j \partial_j^{2} v_j \, dz \, \frac{d\lambda}{\lambda^{1-2p_j}} \\ 
&\leq  C \sum_{j=1}^{n} p_j \int_\delta^1 \int_{\partial Q \cap \{z_j = \pm 1\}} \partial_j v_j \, dz \, \frac{d\lambda}{\lambda^{1-2p_j}} \\
&\leq \sum_{j=1}^{n} C \sqrt{\varepsilon} \int_0^1 \int_{\partial Q \cap \{z_j = \pm 1\}} |u(z,\lambda)|  \, dz \, \frac{d\lambda}{\lambda^{1-p_j}},
\end{align*}
where the final step follows from \eqref{locparbetat.eq}. By the same arguments, we find for all $z \in \overline{Q}$ and all $\lambda \in (0,1)$ that
\begin{align*}
\frac{u(z,\lambda)}{u(z,1)}
&= \exp \biggl( -\int_{\lambda}^{1} \partial_{r} \log u(z,r) \, dr  \biggr) \\
&= \exp \biggl( \sum_{j=1}^{n} p_{j} \int_{\lambda}^{1} \frac{\partial_j^{2} v_j(z,r)}{u(z,r)} \, \frac{dr}{r^{1-2p_j}}  \biggr) \\
& \leq \exp \biggl( \sum_{j=1}^{n} p_{j} \int_{\lambda}^{1} \sqrt{\varepsilon} \, \frac{dr}{r}  \biggr) 
= \lambda^{ - \sqrt{\varepsilon} p  }.
\end{align*}
Combining the previous two estimates gives
\begin{align*}
\inf_{z \in \overline{Q}} \frac{1}{u(z,1)} \fint_{Q} u(z,\delta)- u(z,1) \, dz
\leq C \sqrt{\varepsilon} \biggl(\sum_{j=1}^{n}  \int_0^1  \lambda^{ - \sqrt{\varepsilon} p } \, \frac{d\lambda}{\lambda^{1-p_j}}\biggr).
\end{align*}
The integrals are constants depending on $p_1,\ldots,p_n$ since we assume $\sqrt{\eps} \leq \min_j \frac{p_j}{2p}$. Going back to \eqref{log-in-painfty.eq}, we get the desired bound
\begin{align*}
	\I_1 \leq \log_+(1 + C \sqrt{\varepsilon}) \leq C \sqrt{\varepsilon}
\end{align*}
and the proof of \eqref{weights-characterization.thm.eq1} is complete.

\emph{Step 2: Absolute continuity of $\omega$}. For clarity, let us write $\omega_\delta \coloneqq u(\bigdot,\delta) \, dz$ for the measures with weight function $u(\bigdot,\delta)$. In the limit as $\delta \to 0$, they converge weakly to $\omega$ on $2Q_0$. By classical theory of weights, \eqref{weights-characterization.thm.eq1} implies $\omega_\delta \in A_\infty(2Q_0)$, uniformly in $\delta$, that is, there are constants $c, \gamma > 0$ such that for all balls $Q \subset (3/2)Q_0$, all Borel sets $E \subset Q$ and all $\delta \in (0, r_0)$ we have
\begin{align}\label{weights-characterization.thm.eq2}
	\frac{\omega_\delta(E)}{\omega_\delta(Q)} \leq c \biggl(\frac{|E|}{|Q|}\biggr)^{\gamma},
\end{align}
see for instance \cite[Theorem~IV.2.15]{GCRDF}. By local doubling $\omega_\delta(Q) \leq \omega (2Q) \leq C_\omega \omega(Q)$ for $\delta$ small enough and if $E$ is open, then also $\liminf_{\delta \to0} \omega_\delta(E) \geq \omega(E)$ by weak convergence. Thus, 
\begin{align*}
		\frac{\omega(E)}{\omega(Q)} \leq C_\omega c \biggl(\frac{|E|}{|Q|}\biggr)^{\gamma}.
\end{align*}
This estimate remains valid for all Borel sets $E \subset Q$ by outer regularity of $\omega$. In particular,  $\omega$ is absolutely continuous with respect to Lebesgue measure on $(3/2)Q_0$.

\emph{Step 3: Conclusion}. Let $Q \subset Q_0$. We complete the proof of \eqref{Ainfsingscalefkp.eq} by passing to the limit in \eqref{weights-characterization.thm.eq1} along a sequence $\delta \to 0$. Since $\omega \in L^1((3/2)Q_0)$, we have $u(\bigdot, \delta) \to \omega$ in $L^1(Q)$ and therefore
\begin{align*}
	\lim_{\delta \to 0} \log \biggl( \fint_Q u(z,\delta) \, dz\biggr)  = \log \biggl(\fint_Q \omega(z) \, dz \biggr).
\end{align*}
The limit with logarithm inside the integral is more delicate. Upon passing to a subsequence, we can arrange $u(\bigdot,\delta) \to \omega$ pointwise almost everywhere and hence the same for $\log u(\bigdot,\delta) \to \log \omega$. Moreover, \eqref{Ainfsingscalefkp.eq} implies that $(\fint_Q \log u(z, \delta) \, dz)_\delta$ is bounded and \eqref{weights-characterization.thm.eq2} implies that $(\log u(\bigdot, \delta))_\delta$ is bounded in $BMO((4/3)Q)$, see for instance \cite[Cor.~IV.2.19]{GCRDF}. Due to the John--Nirenberg inequality, $(\log u(\bigdot,\delta))_\delta$ is bounded in $L^2(Q)$. Since boundedness and pointwise convergence imply weak convergence in $L^2(Q)$, we can conclude
\begin{align*}
	\lim_{\delta \to 0} \fint_Q \log u(z,\delta) \, dz = \fint_Q \log \omega(z) \, dz.
\end{align*}
Now, \eqref{weights-characterization.thm.eq1} follows.
\end{proof}
 %%%%%%%%%%%%%%%%%%%%%%%%%%%%%%%%%%%%%%%%%%%%%%%%%%%%%%%%%%%%%%%%
\section{Proofs of the main results}
\label{main results.sec}

\noindent Eventually, we assemble all estimates proved so far in order to obtain Theorem~\ref{main-theorem}.  We start with a slightly more precise local variant from which the rest will follow effortlessly.

\begin{theorem}
\label{main-thm-general.thm}
Let $M_0$ be given.
There are constants $\varepsilon_0 > 0$, $\kappa_0 \geq 40$, $\delta_0 > 0$, $\theta_0 \in (0, 1/4]$, $\gamma>0$ and $C \ge 1$ such that the following hold whenever $\cL$ is $(M_0,\varepsilon)$-parabolic and $\omega \coloneqq \omega^p$ is the $\cL$-parabolic measure with pole at $p \coloneqq a^{+}(t_0,x_0,\kappa\lambda_0)$ for some  $(t_0,x_0,\lambda_0) \in \ree_{+}$, $\kappa \geq \kappa_0$.
\begin{enumerate}
	\item If $\no{\nu_{A,B}}_{\mathcal{C}} \le \delta_0$ and $\theta \leq \theta_0$, then $\omega$ is absolutely continuous with respect to Lebesgue measure on $Q^n(t_0,x_0, 2\theta^2 \lambda_0)$.
	\item Denoting $k = \frac{d\omega}{dt\,dx}$ in the setting of (1), it follows for all $(t,x,\lambda) \in R(t_0,x_0,\lambda_0)$ that
	\begin{align*}
		%\label{eq:main-thm-general}
		\log \left(\fiint_{Q^{n}(t,x,\theta^2\lambda  )} k(t,x) \, dt \, dx \right)  &- \fiint_{Q^{n}(t,x,\theta^2\lambda)} \log k(t,x) \, dt \, dx \\
		&\le C \bigl( \theta^{\gamma} + \sqrt{\no{\nu_{A,B}}_{\mathcal{C}(R(t,x,2\lambda)) } } \bigr)  .
	\end{align*}
\end{enumerate}
\end{theorem}
%
%%	Given $(t,x,\lambda) \in R(t_0,x_0,\lambda_0)$, 
\begin{proof}
Let us first sort out the constants. We take $\varepsilon_0, \kappa_0$ as in Lemma~\ref{non-constant-DLM-main.lem} and $\theta_0, \gamma$ as in Theorem~\ref{green-main-estimate.thm}. By virtue of  Proposition~\ref{whitney-carleson.prop}, we can take the threshold $\delta_0$ such that the assumption $\no{\nu_{A,B}}_{\mathcal{C}}\leq \delta_0$ allows us to use Theorem~\ref{green-main-estimate.thm} for every $(t,x,\lambda) \in R(t_0,x_0,\lambda_0)$ and in particular for $(t_0,x_0, \lambda_0)$.

Both assertions will now follow from Theorem~\ref{weights-characterization.thm} applied to the parabolic measure $\omega$ on $Q_0 \coloneqq Q^{n}(t,x,\theta^2 \lambda)$. 

The local doubling has been discussed in Lemma~\ref{HL-Green-pm.lem} and we fix a split heat kernel with anisotropy $(p_1,\ldots,p_n) = (2,1,\ldots,1)$ and profile supported in $[-3/4,3/4]$. As in Lemma~\ref{parab-measure-main.lem}, we set $u \coloneqq G(p,\bigdot)$ and we also let $\mu_{\beta_u}$ be the same measure as in Theorem~\ref{green-main-estimate.thm}. Then \eqref{ratio-general.eq} yields \eqref{locparalph.eq} with
\begin{align*}
	 \varepsilon \coloneqq C \bigl(\no{\mu_{\beta_u}}_{\mathcal{C}(R(t,x,2\theta^2 \lambda))} + \no{\nu_{A,B}}_{\mathcal{C}(R(t,x,2\theta^2 \lambda))} \bigr) 
\end{align*}
and by averaging the general version of \eqref{ratio-general.eq} with respect to $(t',x',\lambda')$ on a suitable Whitney-type region, we get \eqref{locparbetat.eq} with the same choice of $\varepsilon$. Theorem~\ref{green-main-estimate.thm} guarantees that
\begin{align*}
	\varepsilon \leq C \bigl(\theta^{2\gamma} + \|\nu_{A,B}\|_{\mathcal{C}(R(t,x, 2\lambda))} \bigr) 
\end{align*}
and the proof is complete.
\end{proof}

\begin{proof}[Proof of Theorem~\ref{main-theorem}]
This is mainly a question of re-labeling objects given that 
Theorem~\ref{main-thm-general.thm} holds uniformly in $\theta \leq \theta_0$ and $\kappa \ge \kappa_0$. 

Indeed, given $\delta$, let us take $\theta < \theta_0$ so small that $\theta^\gamma < \sqrt{\delta}$.
Setting $\kappa \coloneqq \kappa_0/\theta^2$ and $\lambda_1 \coloneqq \lambda_0 / \theta^2$, we see that $a^{+}(t_0,x_0,\kappa\lambda_0) = a^{+}(t_0,x_0,\kappa_0\lambda_1)$ and that every parabolic cube $Q \subset Q^{n}(t_0,x_0,\lambda_0)$ is of the form $Q = Q^{n}(t,x,\theta^2 \lambda)$ with $(t,x,\lambda) \in R(t_0,x_0,\lambda_1)$.
Hence, Theorem~\ref{main-thm-general.thm} yields absolute continuity of $\omega$ on $Q^{n}(t_0,x_0, 2 \lambda_0)$ along with the estimate
\[
\log \biggl(\fiint_{Q} k(t,x) \, dt\, dx \biggr)  - \fiint_{Q} \log k(t,x) \, dt \, dx \le C \sqrt{\delta} . \qedhere
\]
\end{proof}
\begin{proof}[Proof of Corollary \ref{main-cor}]
The claim on reverse H\"older classes follows by Theorem~\ref{main-theorem} and \cite[Theorem~9]{Kor2},
taking into account that the John--Nirenberg theorem remains valid in metric spaces with a locally doubling measure. The claim on solvability of the Dirichlet problem then follows from a standard argument that is similar to the elliptic case, see e.g. \cite{NystromLP} or \cite[Lemma~4.19, Section 2]{HL-Mem}.
\end{proof}

\bibliography{DKPBESrefs2}
\bibliographystyle{alpha}

\end{document}